\documentclass[a4paper,12pt]{amsart}

\usepackage{verbatim}
\usepackage{geometry}
\usepackage[utf8]{inputenc}
\usepackage[english]{babel}
\usepackage{amsmath}
\usepackage{amssymb}
\usepackage{amsthm}
\usepackage{mathtools}
\usepackage{hyperref}
\usepackage{tikz}
\usepackage{tikz-cd}
\usepackage{mathrsfs}
\usepackage{quiver}
\usepackage{bm}

\DeclareFontFamily{U}{mathc}{}
\DeclareFontShape{U}{mathc}{m}{it}%
{<->s*[1.03] mathc10}{}

\DeclareMathAlphabet{\customcal}{U}{mathc}{m}{it}

\makeatletter
\DeclareFontEncoding{LS1}{}{}
\DeclareFontSubstitution{LS1}{stix}{m}{n}
\DeclareMathAlphabet{\mathscr}{LS1}{stixscr}{m}{n}
\makeatother

\theoremstyle{plain}
\newtheorem{thm}{Theorem}[section]
\newtheorem{prop}[thm]{Proposition}
\newtheorem{conj}[thm]{Conjecture}
\newtheorem{lemma}[thm]{Lemma}
\newtheorem{coroll}[thm]{Corollary}

\newtheorem*{thm*}{Theorem}

\theoremstyle{definition}
\newtheorem{defn}[thm]{Definition}
\newtheorem*{defn*}{Definition}

\theoremstyle{remark}

\newtheorem{remark}[thm]{\underline{Remark}}

\DeclareMathOperator{\rk}{rk}
\DeclareMathOperator{\st}{\; | \;}
\newcommand{\ST}{\; \middle| \;}

\newcommand{\abs}[1]{\left\lvert #1 \right\rvert}

\newcommand{\ddl}[1]{\frac{\partial #1}{\partial \lambda}}
\newcommand{\hatcalC}{\widehat{\mathcal{C}}}
\newcommand{\legsch}{\mathcal{L}}

\newcommand{\m}{\mathfrak{m}}
\newcommand{\N}{\mathbb{N}}
\renewcommand{\P}{\mathbb{P}}
\newcommand{\Z}{\mathbb{Z}}
\newcommand{\Q}{\mathbb{Q}}
\newcommand{\R}{\mathbb{R}}
\newcommand{\C}{\mathbb{C}}

\renewcommand{\O}[1][]{\mathscr{O}_#1}

\newcommand\restr[2]{#1 \raisebox{-.5ex}{$|$}_{#2}}

\newcommand{\quotes}[1]
{``#1''}

\numberwithin{equation}{section}

\title{Singular intersections in families of abelian varieties}
\author{Nicola Ottolini}
\address{Dipartimento di Matematica, Università degli Studi di Roma “Tor Vergata”, Rome (Italy)}
\email{ottolini@mat.uniroma2.it }
\email{nicolaotto@outlook.it}


\newcounter{TMPenumnbr}

\begin{document}
\maketitle
\vspace{ -4em}
\hspace{6em}
\begin{abstract}
Let $S$ be a smooth irreducible curve defined over $\overline{\Q}$, let $\mathcal{A}$ be an abelian scheme over $S$ and $\mathcal{C}$ a curve inside $\mathcal{A}$, both defined over $\overline{\Q}$.
In this paper we prove that the set of points in which $\mathcal{C}$ intersects proper flat subgroup schemes of $\mathcal{A}$ tangentially is finite.
The crucial case of elliptic curves already follows from a result by Corvaja, Demeio, Masser and Zannier: in this case we provide an alternative proof using the Pila-Zannier method. 
Such a proof may lead to an effective result using an effective point-counting theorem.

This fits in the framework of the so-called problems of unlikely intersections, and can be seen as a variation of the relative Pink conjecture for abelian varieties.
\end{abstract}

\section{Introduction}
\label{Section10}
Let $\lambda:\legsch \to \C \setminus \{0, 1\}$ be the Legendre elliptic scheme,
where for every $\lambda_0 \in \C \setminus \{ 0, 1\}$ the fiber $\legsch_{\lambda_0}$ is an elliptic curve defined by
\begin{equation}
\label{legendre_form_eqn}
Y^2 = X (X - 1)(X - \lambda_0).
\end{equation}
Masser and Zannier \cite{MZ08, MZ10} showed that there are at most finitely
many complex numbers $\lambda_0 \neq 0, 1$ such that the two points
\begin{align*}
(2, \sqrt{2(2-\lambda_0)}),\quad & (3, \sqrt{6(3-\lambda_0)})
\end{align*}
both have finite order on the elliptic curve $\legsch_{\lambda_0}$.
Later, the same authors proved in \cite{MZ12} that one can replace 2 and 3 by any two distinct complex numbers $(\neq 0, 1)$ or even choose
distinct $X$-coordinates $(\neq \lambda)$ defined over an algebraic closure of $\C(\lambda)$, provided that the corresponding points $P_1, P_2$ are not 
identically linearly related, i.e. there are no integers $a_1, a_2 \in \Z$ not both zero such that $a_1 P_1 + a_2 P_2 = O$ identically on $\mathcal{C}$.
This was further generalized by Barroero and Capuano \cite{BarroCapuano16}, who proved 
that, taking an arbitrary number of points defined over $\overline{\Q(\lambda)}$ such that they are not 
identically linearly related,
the set of complex numbers $\lambda_0$ such that they satisfy two independent linear condition with integer coefficient is again finite.

The same authors further generalized this result to 
product of powers of two non-isogenous elliptic schemes \cite{BarroCapuano17} and then to 
general abelian schemes \cite{BarroCapuano20}, also generalizing a result in \cite{MZ20}, where the authors consider only the intersection with the torsion;
the main result is the following. 
\begin{thm}[\cite{BarroCapuano20}, Thm 1.1]
\label{BC16_thm}
Let $S$ be a smooth irreducible curve defined over $\overline{\Q}$.
Consider $\mathcal{A} \to S$ an abelian scheme and $\mathcal{C}$ an
irreducible curve in $\mathcal{A}$ not contained in a proper subgroup scheme of $\mathcal{A}$, even after a finite base change.
Let 
\[
\mathcal{C}^{\{2\} } = \bigcup_{\overset{G \text{ flat}}{\text{codim }G \geq 2} } \mathcal{C} \cap G,
\]
where $G$ runs through all the \emph{flat} subgroup schemes of codimension at least 2, that is, all the subgroup schemes of codimension at least 2 whose irreducible components dominate the base $S$;
then, $\mathcal{C}^{\{2\} }$ is a finite set.
\end{thm}

All these results fit in the framework of very general conjectures, formulated independently by Zilber \cite{Zilber2002} and Bombieri, Masser and Zannier \cite{BMZ99} in the toric case
and by Pink \cite{Pink2005} in the more general setting of mixed Shimura varieties, including classical conjectures such as Manin-Mumford, André-Oort, known as the Zilber–Pink conjectures.
Shimura varieties are a class of algebraic varieties 
including abelian varieties, families of abelian varieties, modular curves and moduli spaces of abelian varieties.
Their structure allows us to define a certain (countable) subset of their subvarieties, called “special subvarieties”.
A general description is quite complicated in general, and can be found in \cite{DeligneShimura}. 
In the case of tori and abelian varieties, the special subvarieties are precisely the irreducible components of algebraic subgroups. 
For families of abelian varieties, the situation is more complicated. 
Flat subgroup schemes, that is, the subgroup schemes that are dominant on the base, are special subvarieties, and actually they are the only ones that are dominant on the base. 
However, there are in general also special subvarieties that are not dominant.
For example, in the case of non-isotrivial elliptic schemes, they are the subgroups of the CM fibers.

Zilber-Pink conjectures involve the intersection of a fixed subvariety $Z$ of a Shimura variety $\mathcal{X}$ with the special subvarieties of $\mathcal{X}$ of codimension bigger than the dimension of $Z$.
The prediction is that, under some natural geometric conditions, such intersection is not Zariski dense in $Z$.

The precise statement for curves in an abelian scheme is the following.
\begin{conj}
Let $S$ be a smooth irreducible curve defined over $\C$.
Consider $\mathcal{A} \to S$ an abelian scheme and $\mathcal{C}$ an
irreducible curve in $\mathcal{A}$, defined over $\C$ and not contained in a proper special subvariety of $\mathcal{A}$.
Consider the set
\[
\mathcal{C}^{[ 2 ] } = \bigcup_{\overset{H \subseteq \mathcal{A}}{\text{codim }H \geq 2} } \mathcal{C} \cap H ,
\]
where $H$ runs through all the special subvarieties of $\mathcal{A}$ of codimension at least 2.
Then $\mathcal{C}^{[2]}$ is a finite set.
\end{conj}

This conjecture is known to hold when $\mathcal{A}$ is a fibered power of an elliptic scheme $\mathcal{E}$, by work of Masser and Zannier \cite{MZ12}, Barroero and Capuano \cite{BarroCapuano16}, Barroero \cite{BarroeroCM} and Barroero and Dill \cite{BarroeroDill}, and in some cases of product of (fibered) powers of two non-isogenous elliptic schemes defined over $\overline{\Q}$, due to Masser and Zannier \cite{MZ14}, Barroero and Capuano \cite{BarroCapuano17} and Ferrigno \cite{Ferrigno24}. The general case, however, is still open.

More details on the Zilber-Pink conjectures on families of abelian varieties may be found in a survey by Capuano \cite{Capuano2023}, and a more in-depth treatment of the topic can be found in books by Zannier \cite{Zannier2012} and Pila \cite{Pila2022}.\\

The geometric idea behind these results is that the intersections we are considering are, for dimensional reasons, expected to be empty.
When this is not the case, we say that we have an \quotes{unlikely intersection}.
The statement of the conjecture can be interpreted as saying that this is indeed a rare phenomenon.

On the other hand, if in the same setting as above one intersects the curve $\mathcal{C}$ with codimension 1 subgroups, then the union of such intersections will be infinite in general.
However, if we restrict ourselves to consider only the points of singular intersection, $\mathcal{C} \cap_{sing} G$, that is the points of $\mathcal{C} \cap G$ where the multiplicity of intersection between $\mathcal{C}$ and a codimension 1 subgroup $G$ is strictly greater then 1, then we may hope to gain back finiteness.
This is because the intersection of a curve with another subvariety of codimension 1 is in general expected to be transverse, that is, not singular.
Geometrically, this corresponds to considering the points where for some $a_1, \dots, a_n$ in the generic endomorphism ring, the locally defined section 
$$a_1 P_1 + \dots + a_n P_n$$
intersects the zero section \emph{tangentially}.

This kind of question has already been investigated, and answered positively, in some cases.
More specifically, the case of curves which are images of sections intersecting tangentially the torsion cosets in a power of an elliptic scheme follows from works of Corvaja, Demejo, Masser and Zannier \cite{CDMZ21}.
Ulmer and Urz{\'u}a \cite{UU20, UU21} give an independent proof of the case of curves that are image of sections in a single power of an elliptic scheme. In the same papers they also show an application to elliptic divisibility sequences.
In a different setting, namely for curves in $\mathbb{G}_m^2$, the question was studied by Marché and Maurin \cite{Maurin23}.
Similar questions have also been studied in the modular setting, where theorems in the same direction have been proven by Spence \cite{Spence2019}, Aslanyan \cite{Aslanyan202} and Ballini \cite[Thm. 6.0.2]{BalliniTesi}, using an Ax-Schanuel type result by Pila and Tsimerman \cite{Pila2016}.

In this paper we extend such result to curves in more general one dimensional families $\mathcal{A} \to S$ of abelian varieties.
To obtain a result analogue to the Zilber-Pink conjectures we want to prove the finiteness of
\[
\mathcal{C}^{[1, sing] } = \bigcup_{\overset{H \text{ special}}{\text{codim }H \geq 1} } \mathcal{C} \cap_{sing} H,
\]
where $H$ runs through the special subvarieties of $\mathcal{A}$, for curves $\mathcal{C}$ that are not contained in any proper special subvariety.
Since by assumption $S$ is a curve, the special subvarieties of $\mathcal{A}$ come in two types: they are either dominant on $S$, in which case they are irreducible components of flat subgroup schemes, or they are contained in some special fibers.
The finiteness of the singular intersections with special fibers is implied by the following statement.
\begin{thm}
\label{singular_intersection_with_fibers}
Let $S$ be a smooth irreducible curve defined over $\C$.
Consider $\mathcal{A} \to S$ be an abelian scheme and let $\mathcal{C}$ be an
irreducible curve in $\mathcal{A}$, not contained in any fiber.
Let 
\[
\mathcal{C}^{\{1, ram\} } = \bigcup_{s \in S } \mathcal{C} \cap_{sing} \mathcal{A}_s.
\]
Then, $\mathcal{C}^{\{1, ram\} }$ is a finite set.
\end{thm}
The proof of this result is the content of Section \ref{SectionFibers}.
The rest of this paper is then dedicated to prove the finiteness of the singular intersections with flat subgroups, with some requirements on the abelian scheme $\mathcal{A}$.
\begin{thm}
\label{abelian_main_thm}
Let $S$ be a smooth irreducible curve defined over $\overline{\Q}$.
Let us consider $\mathcal{A} \to S$ an abelian scheme 
that has no isotrivial elliptic component, that is we cannot write $\mathcal{A} = E \times \mathcal{A}'$ where $E$ is a constant elliptic curve, even after base change.
Let $\mathcal{C}$ be an
irreducible curve in $\mathcal{A}$, dominant over $S$ and not contained in a proper subgroup scheme of $\mathcal{A}$, even after a finite base change.
Let 
\[
\mathcal{C}^{\{1, sing\} } = \bigcup_{\overset{G \text{ flat}}{\text{codim }G \geq 1} } \mathcal{C} \cap_{sing} G,
\]
where $G$ runs through all the flat subgroup schemes of codimension at least 1.\\
Then, $\mathcal{C}^{\{1, sing\} }$ is a finite set.
\end{thm}

We remark that the statement is expected to be true also without the condition of $\mathcal{A}$ not having isotrivial elliptic component, although our techniques require to exclude that case.
Work is in progress with Ballini and Capuano to prove the same finiteness result for tori and product of (constant) elliptic curves. Such a result, combined with the rest of the proof of Theorem \ref{abelian_main_thm}, would allow us to drop such assumption.

To prove Theorem \ref{abelian_main_thm},
similarly to \cite{BarroCapuano20}, we start by showing that we can substitute $\mathcal{A}$ with an isogenous abelian scheme $\mathcal{A}'$.
Therefore we can assume, up to isogeny, that $\mathcal{A}$ is the product $\prod_i \mathcal{A}_i^{n_i}$ of fibered powers of pairwise non-isogenous simple abelian schemes.
The finiteness of the intersection of $\mathcal{C}$ with the flat subgroups of codimension $\geq2$ is proven in \cite{BarroCapuano20}, so we can restrict ourselves to consider only codimension 1 subgroups.
Those will all be of the shape
\[
H = \mathcal{A}_1^{n_1} \times_S \dots \times_S \mathcal{A}_{i-1}^{n_{i-1}} \times_S H_i \times_S \mathcal{A}_{i+1}^{n_{i+1}} \dots \times_S \mathcal{A}_r^{n_r},
\]
where $H_i < \mathcal{A}_i$ is a codimension 1 subgroup of $\mathcal{A}_i$, and $\mathcal{A}_i$ is a power of an elliptic scheme.
Therefore, it is enough to consider the case where $\mathcal{A}$ is itself a power of an elliptic scheme.
By assumption, $\mathcal{A}$ will be non-isotrivial, and adapting some work of Habegger \cite{Habegger2013} we can reduce to the case where $\mathcal{A}$ is a power of the Legendre scheme.
The full reduction is performed in the third section.\\
We are then left to prove the following statement.
\begin{thm}
\label{Main_Thm}
Let $\lambda: \legsch \to S = \P \setminus \{0, 1, \infty \}$ be the Legendre elliptic scheme with fibers $\legsch_\lambda$, and let $\legsch^n$ be its $n$-th fibered power, with coordinates $P_1, \dots, P_n$. 
For $\bm{a} \in \Z^n$ let $\varphi_{\bm{a}}: \legsch^n \to \legsch$ be the fiberwise linear map given by 
\[
\varphi_{\bm{a}}(P_1, \dots, P_n) = a_1 P_1 + \dots + a_n P_n .
\]
Let $\mathcal{C} \subseteq \legsch^n$ be an irreducible curve defined over $\overline{\Q}$ and not contained in a fiber over $S$.
Assume that, for
every $j = 1, \dots , n$,
there is no 
$\bm{a} = (a_1 , \dots , a_n )\in \Z^n\setminus \{0\}$ such that
\begin{equation}
\varphi_{\bm{a}} =  O\circ\lambda,
\end{equation}
identically on $\mathcal{C}$, where $O$ is the zero section of $\legsch$. 
Then, there are at most finitely many $\bm{c} \in \mathcal{C}$ such that there exists $\bm{a} \in \Z^n\setminus \{0\}$ with
\begin{equation}
\label{linear_cond_eqn}
\varphi_{\bm{a}}(\bm{c}) = O(\lambda(\bm{c}))
\end{equation}
and such that $\mathcal{C}$ and $\ker (\varphi_{\bm{a}})$ intersect tangentially at $\bm{c}$.
\end{thm}

This result follows from \cite[Thm. 2.10]{CDMZ21} and the invariance of the problem under base change, proven in Lemma \ref{ab_main_basechange_invariant_lemma}.
However, we provide an independent proof of it following Pila-Zannier method, a strategy introduced by Pila and Zannier in
\cite{PZ08}
to give an alternative proof of the Manin-Mumford conjecture (already proved by Raynaud in \cite{Raynaud83}),
and used in several works in the field.
This method may lead to an effective result for the number of points of singular intersection, using an effective version of Proposition \ref{ominHPprop}, such as \cite[Cor. 4.6]{Binyamini2023}.

To prove finiteness, we prove that $\mathcal{C}^{\{1, sing\}}$ is a set of bounded height and degree, and then conclude by Northcott property.
In particular, the height bound comes from Silverman Specialization Theorem \cite{SilvermanSpec}, while for the degree we compare two estimates for the growth of 
the size of the set of singular intersections of $\mathcal{C}$ with subgroups defined by linear conditions with bounded coefficients.
To get an upper bound, we consider the
elliptic logarithms $z_1, \dots, z_n$ of $P_1,\dots , P_n$ restricted to the curve $\mathcal{C}$, their derivatives with respect to $\lambda$ and the equations
\begin{equation}
\begin{cases}
z_j = u_j f + v_j g ,\\
\ddl{z_j} = \ddl{u_j} f + u_j \ddl{f} +\ddl{v_j} g + v_j \ddl{g} ,
\end{cases}
\end{equation}
for $j = 1, \dots , n$, where $f$ and $g$ are suitably chosen basis elements of the period
lattice of $\legsch_\lambda$. If we consider the coefficients $u_j , v_j, \ddl{u_j},\ddl{v_j} $ as functions of $\lambda$ and restrict
them to a compact set, we obtain a subanalytic surface $\Sigma$ in $\R^{4n}$.
The points where $\mathcal{C}$ has singular intersection with a flat subgroup scheme give rise to two linear relations with integer coefficients, one for the $z_j$s and another one (actually the same) for the derivatives.
Consequently, the points will correspond to points
of $\Sigma$ lying on linear varieties defined by equations whose coefficients are the same as those of the linear (on the elliptic curve) relations that the points satisfy.
Combining an o-minimality result by Habegger and Pila \cite{HabeggerPila} with transcendence of the elliptic logarithms and their derivatives proven by Bertrand \cite{Bertrand1990}, we get that the number of points lying on $\Sigma$ and these linear varieties with coefficients bounded by $T$ grows slower than any power of $T$.\\
On the other side, if a point of degree $d$ lies in $\mathcal{C}^{\{1, sing\}}$, it brings all of its Galois conjugates, and we can prove that all of them need to satisfy a linear equation with coefficients bounded by a constant times a certain power of $d$. Combining this with the upper bound forces the degree to be bounded, concluding the argument.\\

We conclude this introduction with a final remark on Theorem \ref{abelian_main_thm}.
In a similar way to \cite{Maurin23}, the result can be reformulated in an unlikely intersections statement.
Let $\mathcal{X}$ be the dual projectivized tangent bundle of $\mathcal{A}$, that is $\mathcal{X} = \mathbb{P}^*(T\mathcal{A})$. 
This comes with a natural map $\tau$ towards $\mathcal{A}$ whose fibers parametrize lines in the corresponding tangent space.
We will therefore denote by $(\bm{p}, [\Delta])$ the point of $\mathcal{X}_{\bm{p}}$ corresponding to the line $\Delta \subseteq T_{\bm{p}} \mathcal{A}$.
The restriction to the smooth locus of the inclusion $\mathcal{C} \to \mathcal{A}$ lifts to $\tilde{\iota}:\mathcal{C}^{sm} \to \mathcal{X}$ in the following way:
for $\bm{c} \in \mathcal{C}^{sm}$ we have by definition that $T_{\bm{c}}\mathcal{C} \subseteq T_{\bm{c}}\mathcal{A}$ is a line. 
Therefore we define $\tilde{\iota}(\bm{c})= (\bm{c}, [T_{\bm{c}}\mathcal{C}])$, and consider $\tilde{\mathcal{C}}$ as the closure of the image of $\mathcal{C}^{sm}$.
In a similar way, for $H<\mathcal{A}$ a (flat) subgroup scheme we can define $\tilde{H} \subseteq \mathcal{X}$ as 
\[
\tilde{H} = \{(\bm{p}, [\Delta])\st \bm{p} \in H, \Delta \subseteq T_{\bm{p}}H\}.
\]
By this construction, we get that $\tilde{\mathcal{C}}$ is again a curve, while if $H$ has codimension 1 in $\mathcal{A}$, then $\tilde{H}$ has codimension 2 in $\mathcal{X}$.
Theorem \ref{abelian_main_thm} can be restated in this language as an unlikely intersections statement.
\begin{thm}
Let $\mathcal{A} \to S$ be an abelian scheme with no isotrivial elliptic component, that is we cannot write $\mathcal{A} = E \times \mathcal{A}'$ where $E$ is a constant elliptic curve, even after base change. 
Let $\mathcal{C}$ be an irreducible curve in $\mathcal{A}$, dominant over $S$ and not contained in a proper subgroup scheme of $\mathcal{A}$, even after a finite base change, and let $\tilde{\mathcal{C}}$ be as above.
Let 
\[
\tilde{\mathcal{C}}^{[2]} = \bigcup_{\overset{H \text{ flat}}{\text{codim }H \geq 1} } \tilde{\mathcal{C}} \cap \tilde{H},
\]
where $H$ runs through all the flat subgroup schemes of codimension at least 1, and $\tilde{H}$ is constructed from $H$ as above.
Then, $\tilde{\mathcal{C}}^{[2]}$ is a finite set.
\end{thm}

\section{Proof of Theorem \ref{singular_intersection_with_fibers} }
\label{SectionFibers}
We recall the setting. 
Let $S$ be a smooth irreducible curve defined over $\C$;
consider $\pi: \mathcal{A} \to S$ an abelian scheme of relative dimension $g$ and $\mathcal{C}$ an
irreducible curve in $\mathcal{A}$, not contained in any fiber.
Recall that we defined
\[
\mathcal{C}^{\{1, ram\} } = \bigcup_{s \in S } \mathcal{C} \cap_{sing} \mathcal{A}_s.
\]
We want to prove that $\mathcal{C}^{\{1, ram\} }$ is a finite set.
Up to removing finitely many points on $S$ and $\mathcal{C}$ we may assume (and we will) that $\mathcal{C}$ is a smooth curve.

We claim that $\mathcal{C}^{\{1, ram\} }$ is 
the ramification locus of $\restr{\pi}{\mathcal{C}} : \mathcal{C} \to S$.
If this is the case, then the finiteness follows from the finiteness of the ramification locus, \cite[Prop. IV.2.2]{Hartshorne1977}.
Notice that $\restr{\pi}{\mathcal{C}}$ is a finite, separable and surjective map between algebraic curves, so by the same proposition the ramification locus is well defined, and it is the support of the relative differential sheaf $\Omega_{\mathcal{C}/S}$.

Let $\mathcal{I}$ be the sheaf of ideals associated with $\mathcal{C}$ in $\mathcal{A}$.
The immersion $\mathcal{C} \subseteq \mathcal{A}$ induces, by \cite[Prop. II.8.12]{Hartshorne1977}, an exact sequence of sheaves
\[
\frac{\mathcal{I}}{\mathcal{I}^2} \xrightarrow{d_{\mathcal{C}}} \left. \Omega_{\mathcal{A}/S} \right| _\mathcal{C} \to \Omega_{\mathcal{C}/S} \to 0.
\]
By Nakayama's Lemma, the support of $\Omega_{\mathcal{C}/S}$ is given by the closed points $p$ such that $\Omega_{\mathcal{C}/S}\otimes k(p) \neq \{0\}$, where $k(p)$ is the residue field at the point $p$.
Hence, these are the points where the map
\[
d_{\mathcal{C}, p}: \frac{\mathcal{I}}{\mathcal{I}^2}\otimes k(p) \to \left. \Omega_{\mathcal{A}/S} \right| _\mathcal{C} \otimes k(p) = \Omega_{\mathcal{A}/S} \otimes k(p)
\] 
is not surjective.
Since $\mathcal{C}$ is smooth, the dimension of $\frac{\mathcal{I}}{\mathcal{I}^2}\otimes k(p)$, the conormal space of $\mathcal{C}$, as a vector space over $k(p)$ is exactly $\dim(\mathcal{A})-\dim(\mathcal{C}) = g$, by \cite[Thm II.8.17]{Hartshorne1977}.
Moreover, since the $\pi$ is smooth, the dimension of the relative cotangent space $\left. \Omega_{\mathcal{A}/S} \right| _\mathcal{C} \otimes k(p)$ is also $g$, so $d_{\mathcal{C}, p}$ is surjective if and only if it is injective.

The map $d_\mathcal{C}$ is the map induced by $d_{\mathcal{A}/S}:\O{\mathcal{A}} \to \Omega_{\mathcal{A}/S}$, so it quotients via $d_{\m_p}: \m_p / \m_p^2 \to \Omega_{\mathcal{A}/S}$, where $\m_p$ is the (maximal) ideal associated with the closed subvariety $p$ of $\mathcal{A}$.

Let $s = \pi(p)$. We have, by \cite[Prop. II.8.11]{Hartshorne1977}, the exact sequence
\[
\pi^*\Omega_{s/S} \to \Omega_{p/S} \to \Omega_{p/s} \to 0.
\]
Since $s$ is a closed subvariety of $S$, the sheaf $\Omega_{s/S}$ is trivial. Moreover, $\Omega_{p/s}$ is also trivial, since both $p$ and $s$ are $\C$-points.
Hence $\Omega_{p/S}$ is trivial.
Since $p$ is also a closed subvariety of $\mathcal{A}$, we also have the exact sequence
\[
\frac{\m_p}{\m_p^2} \xrightarrow{d_{\m_p}} \left. \Omega_{\mathcal{A}/S} \right| _p \to \Omega_{p/S} \to 0.
\]
Since $\Omega_{p/s}$ is trivial, tensoring with $k(p)$ we have that
$d_{\m_p, p}:\m_p/\m_p^2\otimes k(p) \to \Omega_{\mathcal{A}/S}\otimes k(p)$ is surjective.
Moreover, $\dim_{k(p)} (\m_p/\m_p^2\otimes k(p)) = \dim \mathcal{A} = g+1$, so the kernel of $d_{\m_p, p}$ is one-dimensional.
Therefore, $p\in \mathcal{C}$ is a ramification point for $\restr{\pi}{\mathcal{C}}$ if and only if $\mathcal{I}/\mathcal{I}^2 \otimes k(p) \subseteq \m_p/\m_p^2 \otimes k(p)$ contains the kernel of $d_{\m_p, p}$.

Consider now the fiber $\mathcal{A}_s$ of $\pi$ over $s=\pi(p)$ as a closed subvariety of $\mathcal{A}$ and let $\mathcal{J}$ be the associated sheaf of ideals. We have the exact sequence
\[
\frac{\mathcal{J}}{\mathcal{J}^2} \xrightarrow{d_{s}} \left. \Omega_{\mathcal{A}/S} \right| _{\mathcal{A}_s} \to \Omega_{\mathcal{A}_s/S} \to 0.
\]
Again we tensor by $k(p)$, getting
\[
\frac{\mathcal{J}}{\mathcal{J}^2} \otimes k(p) \xrightarrow{d_{s, p}} \left. \Omega_{\mathcal{A}/S} \right| _{\mathcal{A}_s} \otimes k(p) \to \Omega_{\mathcal{A}_s/S} \otimes k(p) \to 0.
\]
Since $\mathcal{A}_s$ is a smooth fiber, $\Omega_{\mathcal{A}_s/S} \otimes k(p)$ has dimension $g$, so by exactness $d_{s, p}$ is the trivial map, sending everything to zero.
Moreover, we have again that $d_s$ quotients through $\m_p / \m_p^2$. 
We claim that the kernel of $d_{\m_p, p}$ is exactly $\mathcal{J}/\mathcal{J}^2 \otimes k(p) \subseteq \m_p / \m_p^2 \otimes k(p)$.
This is the case, since $\mathcal{J}/\mathcal{J}^2 \otimes k(p)$ is positively dimensional and contained into the kernel of $d_{\m_p, p}$, and we have proven that this kernel is one-dimensional.

Now, recall that $\mathcal{C}$ being tangent to $\mathcal{A}_s$ at $p$ is equivalent to the tangent space of $\mathcal{C}$ being contained in the tangent space of $\mathcal{A}_s$, or, equivalently, that $\mathcal{J}/\mathcal{J}^2 \otimes k(p)$, the conormal space of $\mathcal{A}_s$, is contained in $\mathcal{I}/\mathcal{I}^2 \otimes k(p)$, the conormal space of $\mathcal{C}$.
But, by the previous argument, this is exactly the condition of $d_{\mathcal{C}, p}$ not being injective, hence not surjective, and so $p$ being a ramification point of $\restr{\pi}{\mathcal{C}}$. 

\qed

\section{Reduction to powers of the Legendre scheme}
In this section we will reduce the proof of Theorem \ref{abelian_main_thm} to the case of powers of the Legendre scheme.
Similarly to \cite{BarroCapuano20}, we prove that our problem is invariant under finite base changes and isogenies.
Combining this with Poincaré Reducibility Theorem we reduce to two cases: the first one is the case of subgroups of codimension at least two, which is covered by a theorem of Habegger and Pila \cite{HabeggerPila} in the case of isotrivial components and by the aforementioned result by Barroero and Capuano \cite{BarroCapuano20} for the non-isotrivial components.
The second case is that of codimension one subgroups in powers of an elliptic scheme, which is not isotrivial by assumption.

In the second part of this section we further reduce ourselves to consider powers of the Legendre scheme and flat subgroups defined by a single linear condition with integer coefficients.

This section is inspired by the second sections of \cite{BarroCapuano16} and \cite{BarroCapuano20}, and by \cite{Habegger2013}.\\

For $V, W$ subvarieties of a variety $X$, and $P \in V \cap W$ we denote by $i(P; V, W; X)$ the intersection multiplicity of $V$ and $W$ in $P$.
The main tool we will use in our reduction steps is the following.
\begin{prop}[\cite{Nowak1998}, p.173]
\label{intersection_invariant_etale_morph_prop}
Let $f : X' \to X$ be a dominant morphism between smooth
varieties defined over an algebraically closed field $k$, let $V'$ and $W'$ be two subvarieties
of $X'$ that intersect properly at a closed point $P'$, and let $V := f (V'), W :=
f (W')$ and $ P := f (P')$.
If the morphism $f$ is étale at $P'$, then
\[
i(P'; V', W'; X') = i(P; V, W ; X).
\]
In other words, the intersection multiplicity is invariant under étale morphisms.
\end{prop}
This will help us show that our problem is preserved by moving back and forward through étale morphisms.\\
We also state a corollary, that will be required in Section \ref{Section:proof_of_main}.
\begin{coroll}
\label{inters_invariant_iso_cor}
Let $f: X \to X$ be an isomorphism between smooth varieties over an algebraically closed field $k$, and let $V$ and $W$ be two subvarieties of $X$ invariant under $f$, that intersect properly at a closed point $P$. Then,
\[
i(P; V, W; X) = i(f(P); V, W; X).
\]
In particular, this holds when $k=\overline{\Q}$, the varieties $X$, $V$ and $W$ are defined over a number field and $f$ is induced by a Galois action.
\end{coroll}

We also note that, since we are proving a finiteness result, we are always allowed to replace the base curve by a
non-empty open subset, and we will tacitly do so. 
This allows us to pass from an abelian variety
defined over a function field of a curve to the corresponding abelian scheme over (a non-empty open subset of) the curve.\\

We recall our setting and some definitions. Let $\mathcal{A}$ be an abelian scheme  over a smooth irreducible curve $S$,
where everything is defined over $\overline{\Q}$. We call $\pi : \mathcal{A} \to S$ the structural morphism.
A subgroup scheme $G$ of $\mathcal{A}$ is a closed subvariety, possibly reducible, which contains the
image of the zero section $O: S \to \mathcal{A}$, is mapped to itself by the inversion morphism and such that
the image of $G \times_S G$ under the addition morphism is in $G$. A subgroup scheme $G$ is called \emph{flat}
if $\restr{\pi}{G} : G \to S$ is flat. By \cite[Proposition III.9.7]{Hartshorne1977}, since $S$ has dimension 1, this is equivalent to
require that all irreducible components of $G$ dominate the base curve $S$.

\subsection{Reduction to powers of simple abelian schemes}

We want to show that, in order to prove Theorem \ref{abelian_main_thm}, we can perform finite base
changes and isogenies.

\begin{lemma}
\label{ab_main_basechange_invariant_lemma}
Let $\mathcal{C}$ be as in Theorem \ref{abelian_main_thm}. Let $\mathcal{A}' = \mathcal{A} \times_S S'$ for some finite cover $S' \to S$ and
let $f$ be the projection $\mathcal{A}' \to \mathcal{A}$. Then, if the conclusion of Theorem \ref{abelian_main_thm} holds for all irreducible
components of $f^{ -1} (\mathcal{C})$, it holds for $\mathcal{C}$.
\end{lemma}
\begin{proof}
First, we see that $f$ is flat because it is a fibered product of two flat morphisms $\mathcal{A} \to \mathcal{A}$
and $S' \to S$.
Moreover, since $\mathcal{A}$ and $\mathcal{A}'$ have the same dimension, by \cite[Corollary III.9.6]{Hartshorne1977},
it follows that $f$ is quasi-finite, and finite, since it is also proper. 
Up to substituting $S'$ and $S$ with an open dense subset, we can also assume that $S' \to S$ is unramified, and therefore $\mathcal{A}' \to \mathcal{A}$ as well, as it is a fibered product of unramified maps.
Hence, the projection $\mathcal{A}' \to \mathcal{A}$ is étale.
By \cite[Corollary III.9.6]{Hartshorne1977},
we have that if $X \subseteq \mathcal{A}$ is an irreducible variety dominating $S$, as $f$ is finite and flat, each
component of $f^{ -1} (X)$ is a variety of the same dimension dominating $S' $. 
It is clear now that, if
the assumptions of Theorem \ref{abelian_main_thm} hold for $\mathcal{C}$, then they must hold for all components of $f^{ -1} (\mathcal{C})$.
Finally, the preimage of any point of $\mathcal{C}$ lying in a flat subgroup scheme of $\mathcal{A}$ of codimension at
least 1 must lie in a flat subgroup scheme of the same codimension, namely the preimage. Moreover, by Proposition \ref{intersection_invariant_etale_morph_prop}, the tangency condition is also preserved.
\end{proof}

\begin{lemma}
\label{ab_main_isogeny_invariant_lemma}
Let $\mathcal{A}$ and $\mathcal{A}'$ be abelian schemes over the same irreducible curve $S$, let $\eta$ be the generic point of $S$, and let $f_\eta : \mathcal{A}'_\eta \to \mathcal{A}_\eta$
be an isogeny between the generic fibers defined over $k(S)$. Moreover, let $\mathcal{C} \subseteq \mathcal{A}$ be a curve
satisfying the assumptions of Theorem \ref{abelian_main_thm}. Then, if the claim of Theorem \ref{abelian_main_thm} holds for all
irreducible components of $f^{-1} (\mathcal{C})$, it holds for $\mathcal{C}$.
\end{lemma}
\begin{proof}
By the proof of \cite[Lemma 2.2]{BarroCapuano20}, we have that $f_\eta$ extends to a map $f: \mathcal{A}' \to \mathcal{A}$ that is finite, flat, and is an isogeny on every fiber.
So, to conclude in the same way as the previous lemma, we only need to show that $f$ is also unramified, and thus étale.\\
Since $f$ is finite and $\mathcal{A}$ and $\mathcal{A}'$ are locally noetherian, by \cite[Chapter 4, Lemma 3.20]{Liu2002} $f$ is unramified if and only if, for every $p \in \mathcal{A}$, the fibers $\mathcal{A}'_p$ are finite.
By the same argument, for any $s \in S$, the map $f_s: \mathcal{A}'_s \to \mathcal{A}_s$ is unramified if and only if for any $p \in \mathcal{A}_s$ the fiber $(\mathcal{A}'_s)_p$ is finite.
By properties of the base change, $(\mathcal{A}'_s)_p = \mathcal{A}'_p$, so it is enough to show that $f_s$ is unramified for every $s \in S$.
This follows from \cite[Section II.7, Corollary 1 of Thm. 4, p.74]{Mumford1988} and the fact that $f$ is fiberwise an isogeny.
\end{proof}

Now we can proceed as in \cite{BarroCapuano20} to show that it is enough to prove Theorem \ref{abelian_main_thm} for products of simple abelian schemes.
Let us consider the generic fiber $\mathcal{A}_\eta$ of $\mathcal{A}$ as an abelian variety defined over $k(S)$. 
Since every abelian variety is isogenous to a product of simple abelian varieties (see
for instance \cite[Corollary A.5.1.8]{Hindry2000}),
there are geometrically simple, pairwise non-isogenous 
abelian varieties $B_1 , \dots , B_m$, such that $\mathcal{A}_\eta$ is isogenous to $A' := \prod _i B_i ^{n_i} $. 
Note that the abelian variety $A'$ and the isogeny might not be defined over $k(S)$, but over a finite extension of it, which is $k(S')$
for some irreducible, non-singular curve $S'$ covering $S$. 
However, by Lemma \ref{ab_main_basechange_invariant_lemma}, we can (and do) assume $S' = S$.
For the same reason, we can assume that the endomorphism ring of $A'$ is defined over $k(S)$.
Up to shrinking $S$, the abelian varieties $B_i$ extend to abelian schemes $\mathcal{B}_i \to S$.
We define $\mathcal{A}'$ as $\prod _i \mathcal{B}_i ^{n_i}$, where the product and the powers are intended as fibered product and fibered powers over $S$. 
Lemma \ref{ab_main_isogeny_invariant_lemma} then allows us to reduce the proof of Theorem \ref{abelian_main_thm} to the case of products of simple abelian schemes.\\

We are now going to describe flat subgroup schemes of $\mathcal{A}$, which is a fibered product $\mathcal{A}_1 \times_S \dots \times_S \mathcal{A}_m$, where $\mathcal{A}_i$ is the $n_i $-th fibered power of $\mathcal{B}_i$ and $\mathcal{B}_1 , \dots , \mathcal{B}_m$ are abelian schemes whose
generic fibers $B_i$ are pairwise non-isogenous geometrically simple abelian varieties.
Moreover, we
let $R_i$ be the endomorphism ring of $B_i$, which we can suppose to be defined over $k(S)$.\\
Fix $i_0 $, with $1 \leq i_0 \leq m$. 
For every $\bm{a} = (a_1 , \dots , a_{n_{i_{0}}} ) \in R_{i_0}^{n_{i_0}} $, we have a morphism $\varphi_{\bm{a}} : \mathcal{A}_{i_0} \to \mathcal{B}_{i_0}$ defined by
\[
\varphi_{\bm{a}} (P_1 , \dots , P_{n_{i_{0}}} ) = a_1 P_1 + \dots + a_{n_{i_0}} P_{n_{i_0}} .
\]
In a similar way, square matrices of size $n_{i_{0}}$ with entries in $R_{i_0}$ will encode endomorphisms of $\mathcal{A}_{i_0}$, 
and $m$-tuples in $\prod_i Mat_{n_i} (R_i)$ will encode endomorphisms of $\mathcal{A}$. These endomorphisms form a ring, which we call $R$.\\
If $\alpha \in R$, the kernel of $\alpha$ is the fibered product of $\alpha: \mathcal{A} \to \mathcal{A}$ and the zero section $S \to \mathcal{A}$. We consider it as a closed subscheme of $\mathcal{A}$. It is quite easy to see that $\ker (\alpha)$ is a subgroup scheme.

Let $g_i$ be the relative dimension of $\mathcal{B}_i$ over $S$.
If $\alpha = (\alpha_1 , \dots , \alpha_m ) \in R$, we define the rank $r(\alpha)$ of $\alpha$ to be
$\sum_i \rk(\alpha_i) g_i$, where $\rk(\alpha_i)$ is the matrix rank of $\alpha_i \in Mat_{n_i} (R_i)$.
In practice, $\ker \alpha$ is a closed subscheme of $\mathcal{A}$ obtained by imposing $\rk(\alpha_i)$ independent equations
on each factor $\mathcal{A}_i$.

\begin{lemma}[\cite{BarroCapuano20}, Lemma 2.3]
\label{SubgroupsAreKer_lemma}
Let $G$ be a flat subgroup scheme of $\mathcal{A}$ of codimension $d$ with $1 \leq d \leq n$.
Then, there exists an $\alpha \in R$ of rank $d$ such that $G \subseteq \ker(\alpha)$.
Moreover, $\ker(\alpha)$ is a flat subgroup scheme of $\mathcal{A}$ of codimension $r(\alpha)$.
\end{lemma}
From this lemma, we can deduce that each flat subgroup scheme of $\mathcal{A}$ is contained in a flat
subgroup scheme of the same dimension and of the form
\[ 
G = G_1 \times_S \dots \times_S G_m , 
\]
where, for every $i = 1, \dots , m$, $G_i$ is a flat subgroup scheme of $\mathcal{A}_i $.\\
This implies that any flat subgroup scheme of $\mathcal{A}$ of
codimension at least 1 is contained in a $G = G_1 \times_S \dots \times_S G_m$ of codimension at least 1.
Moreover, up to enlarging $G$, and to possibly increasing its dimension, we can assume that all $G_i = \mathcal{A}_i$ except for one index $i_0$, and that $G_{i_0}$ is a maximal flat subgroup of $\mathcal{A}_{i_0}$,
In this way, $G_{i_0}$, and therefore $G$, has codimension exactly $g_{i_0}$.
By projecting on the factors, we only need to prove Theorem \ref{abelian_main_thm} in the case in which $m=1$, that is, $\mathcal{A}$ is simple.

Since by assumption our abelian scheme does not contain any isotrivial elliptic factor, the case $\mathcal{A}$ simple can be further divided in these three cases:
\begin{enumerate}
\item $g_1 \geq 2$ and $\mathcal{A}$ is isotrivial;
\item $g_1 \geq 2$ and $\mathcal{A}$ is not isotrivial;
\item $g_1 = 1$ and $\mathcal{A}$ is not isotrivial.
\end{enumerate}

In the first two cases, $G$ needs to have codimension $\ge 2$, so the claim follows from \cite{HabeggerPila} in the isotrivial case and from \cite{BarroCapuano20} in the non isotrivial one.
Therefore, we are only left to prove the third case.

We point out that, in the case that $\mathcal{A}$ is an isotrivial family of elliptic curves, we still believe the conclusion to be true, although our method does not apply directly.
If we are intersecting with $G$ of codimension $\ge 2$, this follows from \cite{HabeggerPila} and from previous works of Rémond and Viada \cite{Remond2003}, Viada \cite{Viada2008} and Galateau \cite{Galateau2010}. 
The case of $G$ of codimension 1, however, is still open, and is part of a work in progress, using different methods, with Ballini and Capuano.

\subsection{Reduction to the Legendre scheme}

In the last subsection we showed that, to prove Theorem \ref{abelian_main_thm}, it is enough to prove the following special case:
\begin{thm}
\label{Geometric_Main_Thm}
Let $S$ be an irreducible and smooth curve and let $\mathcal{E} \to S$ be a non isotrivial elliptic scheme over $S$ both defined over $\overline{\Q}$.
Let $\mathcal{A} = \mathcal{E}^n$ be the $n$-th fibered power of $\mathcal{E}$ over $S$,
and let $\mathcal{C}$ be an irreducible curve in $\mathcal{A}$ defined over $\overline{\Q}$ and not contained in any proper subgroup schemes of $\mathcal{A}$.
Let 
\[
\mathcal{C}^{\{1, sing\} } = \bigcup_{\overset{G < \mathcal{A} \text{ flat}}{\text{codim }G \geq 1} } \mathcal{C} \cap_{sing} G.
\]
Then, $\mathcal{C}^{\{ 1, sing\} }$ is a finite set.
\end{thm}

We will show that this is implied by Theorem \ref{Main_Thm}.
To do so, we will need to use some results from a work of Habegger \cite{Habegger2013}.\\
Consider the Legendre family defined by 
\[
Y^2 = X (X - 1)(X - \lambda).
\]
This gives an example of a non isotrivial
elliptic scheme, which we call the Legendre elliptic scheme and denote by $\legsch$, over the modular curve $Y (2) = \mathbb{P}^1 \setminus \{0, 1, \infty\}$.
We write $\mathcal{A} _L$ for the $n$-fold fibered power of $\legsch$ and $\pi_L$ for the structural morphism $\pi_L : \mathcal{A}_L \to Y(2)$.

\begin{lemma}[{\cite[Lemma 5.4]{Habegger2013}}]
\label{habegger_diagram_to_Legendre_lemma}
Let $\mathcal{A}$ be as above. After possibly replacing $S$ by a Zariski open, nonempty subset, there exists an irreducible, non-singular,
 curve $S'$ defined over $\overline{\Q}$ such that we have a commutative diagram
\[\begin{tikzcd}
	{\mathcal{A}} & {\mathcal{A}'} & {\mathcal{A}_L} \\
	S & {S'} & {Y(2)}
	\arrow["f"', from=1-2, to=1-1]
	\arrow["l", from=2-2, to=2-1]
	\arrow["{\lambda}"', from=2-2, to=2-3]
	\arrow["e", from=1-2, to=1-3]
	\arrow["{\pi}"', from=1-1, to=2-1]
	\arrow["{\pi_L}", from=1-3, to=2-3]
	\arrow[from=1-2, to=2-2]
\end{tikzcd}\]
where $l$ is finite, $\lambda$ is quasi-finite, $\mathcal{A}'$ is the abelian scheme $\mathcal{A} \times_S S'$, $f$ is finite and
flat and $e$ is quasi-finite and flat. Moreover, the restriction of $f$ and $e$ to any fiber of
$\mathcal{A}' \to S'$ is an isomorphism of abelian varieties.
\end{lemma}
The following lemma, already stated in \cite{BarroCapuano16}, follows from the first part of Lemma 5.5 of \cite{Habegger2013}.
\begin{lemma}[{\cite[Lemma 2.4]{BarroCapuano16}}]
\label{habegger_preserves_info_lemma}
If $G$ is a flat subgroup scheme of $\mathcal{A}$, then $e( f ^{-1} (G))$ is a flat subgroup
scheme of $\mathcal{A}_L$ of the same dimension. Moreover, let $X$ be a subvariety of $\mathcal{A}$
dominating $S$ and not contained in a proper flat subgroup scheme of $\mathcal{A}$, let $X''$ be an
irreducible component of $f^{ -1} (X )$ and let $X'$ be the Zariski closure of $e(X'')$ in $\mathcal{A}_L$.
Then, $X'$ has the same dimension as $X $, dominates $Y (2)$ and is not contained in a
proper flat subgroup scheme of $\mathcal{A} _L $.
\end{lemma}

Using these two lemmas we can now prove how Theorem \ref{Geometric_Main_Thm} follows from Theorem \ref{Main_Thm}.

\begin{proof}[Proof of Theorem \ref{Geometric_Main_Thm} assuming Theorem \ref{Main_Thm}]
First, note that, since $\mathcal{C}$ is not contained in any proper subgroup scheme, it is in particular not contained in any flat subgroup scheme. Moreover, it is also not contained in any fiber, as each fiber with the zero section forms a subgroup scheme.

To ease the notation, we also extend the definition of $\mathcal{C}^{\{1, sing\}}$ to non-irreducible curves as follows:
\[
\mathcal{C}^{\{1, sing\} } := \bigcup_{\mathcal{D} \subseteq \mathcal{C} \text{ irred}}\bigcup_{\overset{G \text{ flat}}{\text{codim }G \geq 1} } \mathcal{D} \cap_{sing} G
\]
where $\mathcal{D}$ runs through the irreducible components of $\mathcal{C}$.

Keeping the notation from Lemma \ref{habegger_diagram_to_Legendre_lemma}, since $e$ is quasi-finite, if $e(f^{-1}(\mathcal{C}^{\{1, sing\}}))$ is finite, then $\mathcal{C}^{\{1, sing\}}$ is finite. 
By Lemma \ref{habegger_preserves_info_lemma} and by Proposition \ref{intersection_invariant_etale_morph_prop}, we have that
\[
\mathcal{Z} := e(f^{-1}(\mathcal{C}^{\{1, sing\}})) \setminus e(f^{-1}(\mathcal{C}))^{\{1, sing\}}
\]
is contained in the fibers over the branch locus of $\lambda$ and over the image by $\lambda$ of the ramification locus of $l$. 
Hence, $\mathcal{Z}$ is contained in the intersection between the dominant curve $e(f^{-1}(\mathcal{C}))$ and a finite number of fibers of $\pi_L$.
Then, it is a finite set, as by Lemma \ref{habegger_preserves_info_lemma} no irreducible component of $e(f^{-1}(\mathcal{C}))$ can be contained in a fiber of $\pi_L$.

It is then enough to prove that $e(f^{-1}(\mathcal{C}))^{\{1, sing\}}$ is a finite set, since $e(f^{-1}(\mathcal{C}^{\{1, sing\}})) \subseteq \mathcal{Z} \cup e(f^{-1}(\mathcal{C}))^{\{1, sing\}}$. \\
Let us consider the Zariski closure $\mathcal{C}'$ of $e(\mathcal{C}'')$ for an irreducible component $\mathcal{C}''$ of $f^{-1}(\mathcal{C})$.
Then, $\mathcal{C}'$ is irreducible, and by Lemma \ref{habegger_preserves_info_lemma}, $\mathcal{C}'$ is a curve in $\mathcal{A}_L$ dominating $Y(2)$ and not contained in a proper flat subgroup scheme. 
Moreover, $e(f^{-1}(\mathcal{C}))$ is a finite union of such $\mathcal{C}'$s.
Proving that $(\mathcal{C}')^{\{1, sing\} } $
is finite will make us conclude the proof.
By Lemma \ref{SubgroupsAreKer_lemma}, each flat subgroup scheme of codimension at least 1 of $\mathcal{A}_L$ is
contained in $\ker(\varphi_{\bm{a}})$ for some $\bm{a} \in \Z^n $. So, it is enough to show that $\bigcup \mathcal{C}' \cap_{sing} \ker(\varphi_{\bm{a}})$ is finite, where the union is taken over all $\bm{a} \in \Z^n \setminus \{ 0 \}$. The claim now follows by applying Theorem \ref{Main_Thm} since $\mathcal{C}'$ is not contained in a proper flat subgroup scheme.
\end{proof}

\section{O-minimality and point counting}
\label{Section20}

In the proof of Theorem \ref{Main_Thm} we will use some results on point counting due to Habegger and Pila \cite{HabeggerPila}. In order to state the result, we need to first introduce some general properties of o-minimal structures.
For a more in-depth treatment we refer to \cite{Dries98} and \cite{Dries96}.
\begin{defn}
A structure is a sequence $\mathcal{S} = (\mathcal{S}_N), N \geq 1$, where each $\mathcal{S}_N$ is a collection of
subsets of $\R^N$ such that, for each $N , M \geq 1$:

\begin{enumerate}
\item $\mathcal{S}_N$ is a boolean algebra (under the usual set-theoretic operations);
\item $\mathcal{S}_N$ contains every semi-algebraic subset of $\R^N$ ;
\item if $A \in \mathcal{S}_N$ and $B \in \mathcal{S}_M$, then $A \times B \in \mathcal{S}_{N + M}$ ;
\item if $A\in \mathcal{S}_{N + M}$, then $\pi (A) \in \mathcal{S}_N$, where $\pi: \R^{N + M} \to \R^N$ is the projection onto the first $N$ coordinates.
\setcounter{TMPenumnbr}{\value{enumi}}
\end{enumerate}
If $\mathcal{S}$ is a structure and, in addition,
\begin{enumerate}
\setcounter{enumi}{\value{TMPenumnbr}}
\item $\mathcal{S}_1$ consists of all finite union of open intervals and points,
\end{enumerate}
then $\mathcal{S}$ is called an \emph{o-minimal structure}.\\

Given a structure $\mathcal{S}$, we say that $S \subseteq \R^N$ is a \emph{definable set} if $S \in \mathcal{S}_N$.

\end{defn}

Let $U \in \R^{M + N}$. For $t_0 \in \R^M$, we set $U_{t_0} = \{x \in \R^N \st (t_0, x) \in U \}$ and call $U$ a \emph{family of subsets}
of $\R^N$, while $U_{t_0}$ is called the \emph{fiber} of $U$ above $t_0$. If $U$ is a definable set, then we call it a \emph{definable
family}, and one can see that the fibers $U_{t_0}$ are definable sets too. Let $S \subseteq \R^N$ and $f : S \to \R^M$ be a
function. We call $f$ a \emph{definable function} if its graph $\{(x, y) \in S\times\R^M \st y = f (x)\}$ is a definable
set. It is not hard to see that images and preimages of definable sets via definable functions are still
definable.
We recall some other properties of definable functions; for a reference see for example \cite[B.7]{Dries96}.
\begin{lemma}
\label{ominimal_derivative_lemma}
Let $U \subseteq \R^n$, and let $f: U \to \R^m$ be a definable function in a structure $\mathcal{S}$. Then:
\begin{enumerate}
\item the set $\{ x \in U \st f \text{ is continuous at } x \}$ is definable in $\mathcal{S}$;
\item the set $U' := \{ x \in U^\circ \st f \text{ is differentiable at } x \}$ is definable in $\mathcal{S}$;
\item the derivative $Df: U' \to \R^{m \times n} \cong \R^{mn}$ is definable in $\mathcal{S}$;
\item the set of all $x \in U'$ such that $f$ is $C^1$ on an open neighborhood of $x$ contained in $U$ is definable in $\mathcal{S}$;
\item if $f: U \to R^m$ is $C^1$, then the set $\{ x \in U \st \rk \left(f(x) \right) = i \}$ is definable in $\mathcal{S}$ for every $i = 1, \dots, m$.
\end{enumerate}
\end{lemma}

There are many examples of o-minimal structures, see \cite{Dries98}. In this article, we are mainly interested
in the structure of globally subanalytic sets, usually denoted by $\R_{\mathrm{an}}$, and the one obtained adjoining also the graph of the real exponential function, usually denoted by $\R_{\mathrm{an}, \mathrm{exp}}$. 
The o-minimality of this last structure was proved by van
den Dries and Miller \cite{Dries94}. 
We are not going to give a detailed description of these structures as it is enough for us to know that, if $D\subseteq \R^N$ is a compact
definable set, $I$ is an open neighborhood of $D$ and $f : I \to \R^M$ is an analytic function, then $f (D)$ is also
definable.\\

We now fix an o-minimal structure $\mathcal{S}$.
\begin{prop}[{\cite[4.4]{Dries96}}]
\label{Def_Family_uniform_bound_prop}
Let $U$ be a definable family; then, there exists a positive integer $\gamma$ such
that each fiber of $U$ has at most $\gamma$ connected components.
\end{prop}

We are going to use a result from \cite{HabeggerPila}. For this, we need to define the height of a rational point.
The height used in \cite{HabeggerPila} is not the usual projective Weil height, but a coordinatewise affine height.
If $a/b$ is a rational number written in lowest terms, then $H (a/ b) = \max \{ \abs{a} , \abs{b} \}$ and, for a
$N$-tuple $(a_1, \dots, a_N ) \in \Q^N$, we set $H (a_1, \dots, a_N ) = \max _{i} H (a_i)$. For a family $Z \subseteq \R^{M_1+ M_2 + N}$,
a positive real number $T$ and $t \in \R^{M_1}$, we define

\begin{equation}
\label{height_stratification_defn_eqn}
Z_t^{\sim} (\Q , T ) = 
\{(y , z) \in Z _t \st y \in \Q ^{M_2} ,H(y) \leq T \}.
\end{equation}

By $\pi_1$ and $\pi_2$, we denote the projections of $Z_t$ to the first $M_2$ and the last $N$ coordinates,
respectively.

\begin{prop}[\cite{HabeggerPila}, Corollary 7.2]
\label{ominHPprop}
Let $Z \subseteq \R^{M_1+ M_2 + N}$ be a definable family. For every $\varepsilon > 0$,
there exists a constant $c = c (Z, \varepsilon)$ with the following property. Fix $t \in \R^{M_1}$ and $T \geq 1$. 
If $\,\abs{\pi_2 (\Gamma)} > cT^\varepsilon$ for some $\Gamma \subseteq Z_t^{\sim}(\Q , T)$, then there exists a continuous definable function
$\delta: [0, 1] \to Z_t$ such that
\begin{enumerate}
\item the composition $\pi_1 \circ \delta: [0, 1] \to \R^{M_2}$ is semi-algebraic and its restriction to $(0, 1)$ is real
analytic;
\item the composition $\pi_2 \circ \delta: [0, 1] \to \R^N$ is non-constant;
\item we have $\pi_2 (\delta (0)) \in \pi_2 (\Gamma)$.
\end{enumerate}

\end{prop}

\begin{remark}
For families that are definable using only Pfaffian functions, such as the Weierstrass $\wp$ function, the constant $c$ of the previous proposition can be bounded by a polynomial of effectively bounded degree, due to work of Binyamini, Jones, Schmidt and Thomas \cite[Cor. 4.6]{Binyamini2023}.
\end{remark}

\section{Periods and elliptic logarithms}
\label{Section30}

In this section, we recall some tools and results from \cite{MZ12} and \cite{BarroCapuano16} regarding the elliptic logarithms.

We start by recalling our setting. 
We defined $\lambda: \legsch \to Y(2) = \P \setminus \{0, 1, \infty \}$ to be the Legendre elliptic scheme, as a subscheme of $\P^2 \times Y(2)$, with fibers $\legsch_\lambda$ defined in affine coordinates by the equation
\[
Y^2 = X (X - 1)(X - \lambda).
\]
It naturally comes with functions $X$ and $Y$ defined on it.
We denote by $\legsch^n$ its $n$-th fibered power, with coordinates $P_1, \dots, P_n$, and with corresponding functions $x_1, \dots, x_n$ and $y_1, \dots, y_n$.

We have an irreducible curve $\mathcal{C}$, defined over $\overline{\Q}$ and embedded in $\legsch^n$, not contained in a fiber over $Y(2)$, and not contained in any proper flat subgroup of $\legsch^n$.\\

Let $\hatcalC$ be the subset of points $\bm{c} \in \mathcal{C}$ such that 
for every $j = 1,\dots, n$ the values $x_j (\bm{c})$ are different from $0, 1, \infty$ and $\lambda(\bm{c})$,
and such that $\bm{c}$ is not a singular point or a point on which
the differential of $\lambda$ vanishes. Note that, in this way, we exclude finitely many $\bm{c} \in \mathcal{C}$, and these are algebraic
points of $\mathcal{C} $. Moreover, on $\hatcalC $, the coordinate function $\lambda$ has a local
inverse everywhere.\\

Fix a point $\bm{c}_\ast \in \hatcalC$ and a (analytic) neighborhood $N_{\bm{c}_\ast}$ on which $\lambda$ has an inverse.
On $N_{\bm{c}_\ast}$ it is possible to define analytic functions $f, g$ that form a basis of the local system of periods $\Pi_{\mathcal{E}_\lambda}$ (i.e. $\mathcal{E}_{\lambda(\bm{c})} \cong \C/ f(\bm{c}) \Z + g(\bm{c}) \Z$ for any $\bm{c} \in N_{\bm{c}_\ast}$); we can take these functions as analytic continuations of hypergeometric functions.

Moreover we can define $z_1, \dots, z_n$, again analytic on $N_{\bm{c}_\ast}$, such that $\exp_ {\lambda(\bm{c})} (z_j (\bm{c})) = P_j (\bm{c})$, for any $\bm{c} \in N_{\bm{c}_\ast}$. In other words, they are elliptic logarithms for the $P_i$'s.
The explicit construction of these maps can be found in \cite{MZ12} and \cite{BarroCapuano16}.

\subsection{A functional transcendence result}

Keeping the same notation as the previous subsection, we prove the following transcendence result, that will be needed later.

\begin{lemma}
\label{funcTrascLemma_fixed_point}
The functions $z _1 , \dots , z _n, \ddl{z_1}, \dots, \ddl{z_n}$ are algebraically independent over 
$\C\left(\lambda, f, g, \ddl{f}, \ddl{g}\right)$ on $N_\ast$.

\end{lemma}
\begin{proof}
The functions $z _1 , \dots , z _n , f , g$ are analytic functions of $\lambda$, linearly 
independent over $\Z$. Indeed, a relation $a_1 z _1 + \dots + a_n z _n = a_{n+1} f + a_{n+2} g$, with integer
coefficients, would map via $\exp_\lambda$ to a relation of the form \eqref{linear_cond_eqn} on $N_\ast$,
and therefore
on all of $\mathcal{C}$, which cannot hold by the hypothesis of the theorem.
Moreover, if $\wp_\lambda$ is
the Weierstrass $\wp$-function associated with $\Pi_{\mathcal{E}_\lambda} = f(\lambda) \Z + g(\lambda) \Z$, the $\wp_\lambda (z _i )$ are algebraic functions of
$\lambda$ because $\wp_\lambda (z _j ) = x_j - \frac{1}{3} (\lambda + 1)$ (see \cite[(3.8), p. 1683]{MZ10}).
Therefore, applying \cite[Théorème 5, p. 136]{Bertrand1990} we have that $z_1, \dots, z_n, \ddl{z_1}, \dots, \ddl{z_n}$ are algebraically independent over $\C\left(\lambda, f, g, \ddl{f}\right)$.\\
To conclude we will prove that $\ddl{g} \in \C\left(\lambda, f, g, \ddl{f}\right)$, and so actually $\C\left(\lambda, f, g, \ddl{f}\right)=\C\left(\lambda, f, g, \ddl{f}, \ddl{g}\right)$. This part of the proof was suggested to us by Daniel Bertrand.\\
Since $f$ and $g$ are hypergeometric series,
they satisfy the hypergeometric differential equation (see \cite[p. 182]{Husemoeller1987}):
 
\[
\lambda(1-\lambda) \frac{d ^2 \omega}{d\lambda^2} + (1-2\lambda) \frac{d \omega}{d \lambda} - \frac{1}{4} \omega = 0.
\]

Since they are linearly independent, they form a basis of the space of solution of such a differential equation.
We can therefore consider the Wronskian associated with them:
\[
W(\lambda) := f \frac{ \partial g}{\partial \lambda} - g \frac{\partial f}{\partial \lambda}.
\]
By explicit computation, we get that
\[
\frac{d W}{d \lambda} = - \frac{2 \lambda - 1}{\lambda(\lambda-1)} W;
\]
solving this differential equation gives $W = k (\lambda^2-\lambda)^{-1}$ for some $k \in \C$, hence $W \in \C (\lambda)$.
We conclude that
\[
\ddl{g} = \left( W + g \ddl{f} \right) \frac{1}{f} \in \C \left( \lambda, f, g, \ddl{f} \right) ,
\]
proving the claim.
\end{proof}

We would like now to extend our functions $f , g, z _1 , \dots , z _n$ on
$\widehat{\mathcal{C}}$.
If $\bm{c} \in \widehat{\mathcal{C}}$, 
one can continue $f$ and $g$ to a neighborhood $N_{\bm{c}}$ of $\bm{c}$. 
In fact, if we
choose $\bm{c} \in \widehat{\mathcal{C}}$ and a path from $\bm{c}_\ast$ to $\bm{c}$ lying in 
$\widehat{\mathcal{C}}$, we can easily continue $f$ and $g$
along the path.\\

To continue $z _j$ from a point $\bm{c}_\ast$ to a $\bm{c}$ in 
$\widehat{\mathcal{C}}$, it is sufficient to verify that if $N_1$ and $N_2$
are two open small subsets in 
$\widehat{\mathcal{C}}$, with $N_1 \cap N_2$ connected, and if $z _j$ has analytic
definitions $z'_j$ on $N_1$ and $z'' _j$ on $N_2$, then $z _j$ has an analytic definition on the union
$N_1 \cup N_2$.
But we saw that $\exp_\lambda (z _j ) = P_j$ for every $j = 1, \dots, n$; hence on $N_1 \cap N_2$
we have $\exp_\lambda (z'_j ) = \exp_\lambda (z''_j )$. This means that there exist rational integers $u, v$
with $z''_j = z'_j + u f + v g$ on this intersection, and they must be constant there. 
Hence it is enough to change $z ''_j$ to $z ''_j - u f - v g$ on $N_2$.\\
Using the same path, it is clear that we can continue the function $(f, g, z _1 , \dots , z _n )$
from a small neighborhood of $\bm{c}_\ast$ to a small neighborhood $N_{\bm{c}} \subseteq \widehat{\mathcal{C}}$ of $\bm{c}$, and that the
obtained function $( f^{\bm{c}} , g ^{\bm{c}} , z _1 ^{\bm{c}} , \dots , z _n ^{\bm{c}} )$ is analytic on $N_{\bm{c}}$. Moreover, the functions
preserve algebraic independence, as the following lemma shows.

\begin{lemma}
\label{funcTrascLemma}
The functions $z _1 ^{\bm{c}} , \dots , z _n ^{\bm{c}} ,
\ddl{z _1 ^{\bm{c}}} , \dots , \ddl{z _n ^{\bm{c}}}$ are algebraically independent over $\C\left(\lambda, f ^{\bm{c}} , g ^{\bm{c}}, \ddl{f ^{\bm{c}}} , \ddl{g ^{\bm{c}}} \right)$
on $N_{\bm{c}}$.
\end{lemma}
\begin{proof}
Any algebraic relation can be continued to a neighborhood $N_\ast$ of some
$\bm{c}_\ast \in \lambda^{-1} (\Lambda)$, contradicting Lemma \ref{funcTrascLemma_fixed_point}.
\end{proof}
Furthermore, the lattice $\Pi_{\mathcal{E}_\lambda}$ is still generated by $f ^{\bm{c}}$ and $g ^{\bm{c}}$ on $N_{\bm{c}}$; see Lemma 6.1
of \cite{MZ12} or Lemma 4.1 of \cite{MZ10}.
Now fix $\bm{c} \in \widehat{\mathcal{C}}$ and $N_{\bm{c}} \subseteq \widehat{\mathcal{C}}$.
Since we are avoiding singular points and points on
which the differential of $\lambda$ vanishes, $\lambda$ gives an analytic isomorphism $\lambda : N_{\bm{c}} \to \lambda(N_{\bm{c}} )$.
Therefore, we can view $z _1^{\bm{c}} , \dots , z _n^{\bm{c}} , f ^{\bm{c}} , g ^{\bm{c}}$ as analytic functions on $\lambda(N_{\bm{c}} )$.

\section{The main estimate}
\label{SectionMainEstimate}
Let us fix a $\bm{c}_\ast \in \widehat{\mathcal{C}}$  and a neighborhood $N_{\bm{c}_\ast}$ of $\bm{c}_\ast$ in $\widehat{\mathcal{C}}$.
Moreover, let us fix a closed disc $D_{\bm{c}_\ast}$ inside $\lambda (N_{\bm{c}_\ast})$, 
centered in $\lambda(\bm{c}_\ast)$.

In Section \ref{Section30} we defined the analytic
functions $f , g$ and $z_i$ on a small neighborhood of $\bm{c}_\ast$, which we may assume to be $N_{\bm{c}_\ast}$ up to shrinking, as a basis for the local system of periods of $\legsch_\lambda$ and elliptic logarithms of the $P_j$.
For the rest of this section, we suppress the dependence on $\bm{c}_\ast$ and $D$ in the notation, since they are fixed.\\
We use Vinogradov’s $\ll$ notation. The implied constants will always depend on $D$.\\

Given $\bm{a} \in \Z^n \setminus \{0\}$, we define $D(\bm{a})$ to be the set of $\lambda \in D$ such that 
\begin{equation}
\label{complex_tangency_condition_eqn}
\left\lbrace
\begin{aligned} 
&\sum_{i=1}^{n} a_i z_i(\lambda) = a_{n+1} f(\lambda) + a_{n+2} g(\lambda), \\
&\sum_{i=1}^{n} a_i \ddl{z_i} (\lambda) = a_{n+1} \ddl{f}(\lambda) + a_{n+2} \ddl{g}(\lambda),
\end{aligned}
\right.
\end{equation} 
for some $a_{n+1}, a_{n+2} \in \Z$.
Notice that $\lambda(\bm{c}) \in D(\bm{a})$ corresponds to $\mathcal{C}$ being tangent to the subgroup $\sum a_i P_i = O$ in $\bm{c}$.\\

For a vector of integers $\bm{a}$ we denote by $\abs{\bm{a}} = \max \{\, \abs{a_1} , \dots , \abs{a_n}\}$.

\begin{prop}
\label{MainEstProp}
Under the hypotheses of Theorem \ref{Main_Thm}, for every $\varepsilon > 0$ we have
$\abs{D(\bm{a})} \ll_\varepsilon \abs{\bm{a}}^\varepsilon$, for every non-zero $\bm{a} \in \Z^n$.
\end{prop}
Before proving this estimate, let us collect a
few more preliminary facts.\\

Let us define
\[
 \Delta = f \overline{g} - \overline{f} g ,
\]
which does not vanish on $D$, since $f (\lambda)$ and $g (\lambda)$ are $\R$-linearly independent for every $\lambda \in D$.
Moreover, for every $j=1, \dots, n$ let $u_j, v_j$ be the functions from $D$ to $\C$ defined by
\begin{align*}
u_j = \frac{z_j \overline{g} - \overline{z_j} g}{\Delta} , \qquad 
v_j = \frac{z_j \overline{f} - \overline{z_j} f}{\Delta}. 
\end{align*}
One can easily check that these functions are real-valued and, furthermore, that we have
\[
z_j = u_j f + v_j g;
\]
$u_j$ and $v_j$ are therefore the so-called \emph{Betti coordinates} of $P_j$.
Taking derivatives on both sides yields
\[
\ddl{z_j} = \ddl{u_j} f  + \ddl{v_j} g + u_j \ddl{f} + v_j \ddl{g}.
\]
If we view $D$ as a subsets of $\R^2 $, then $u _j, v _j, \ddl{u_j}, \ddl{v_j}$ are real analytic functions on a neighborhood of $D$.\\

Define now $\Theta : D \to \R^{4n}$ by
\[
\lambda \mapsto \left(u_1(\lambda), v_1(\lambda), \dots, u_n (\lambda), v_n(\lambda), \ddl{u_1}(\lambda), \ddl{v_1}(\lambda), \dots, \ddl{u_n}(\lambda), \ddl{v_n}(\lambda)\right)
\]
and set $\Sigma = \Theta(D)$. Since $\Theta$ is analytic and $D$ is a closed disc, we have that $\Sigma$ is definable in $\R_{\mathrm{an}} $, as image of a definable set under a definable function.

\begin{lemma}[\cite{BarroCapuano17}, Lemma 4.2]
\label{uniform_bound_linear_rel_lemma}
Under the hypotheses of Theorem \ref{Main_Thm}, there exists a constant $\gamma$ (depending
only on $D$) such that, for every choice of integers 
$a_1, \dots, a_{n + 2}$, not all zero, the number of $\lambda$ in $D$
with
\begin{equation}
\label{Single_Linear_cond_Eqn}
a_1 z_1 (\lambda ) + \dots + a _n z_n (\lambda ) = a _{n + 1} f (\lambda ) + a _{n + 2} g (\lambda ),
\end{equation}
is at most $\gamma$.
\end{lemma}

In what follows $(u_1, v_1, \dots, r_n, s_n)$ will denote the coordinates in $\R^{4n}$.\\
For $T > 0$ , let us denote by $\Sigma^{(1)} (\bm{a}, T )$ the set of points of $\Sigma$ of coordinates $(u_1, v_1, \dots, r_n, s_n)$ such that
there exist $a_{n + 1}, a_{n + 2} \in \Z \cap  [- T , T ]$ with
\begin{equation}
\label{real_linear_condition_eqn}
\left\{
\begin{aligned}
a_1 u_1 + \dots + a_n u_n &= a_{n+1}, \\
a_1 v_1 + \dots + a_n v_n &= a_{n+2}, \\
a_1 r_1 + \dots + a_n r_n &= 0, \\
a_1 s_1 + \dots + a_n s_n &= 0.
\end{aligned}
\right.
\end{equation}

\begin{lemma}
\label{Habegger_bound_lemma}
Under the hypotheses of Theorem \ref{Main_Thm}, for every $\varepsilon > 0$, we have
\[
\abs{\Sigma^{(1)}(\bm{a}, T)} \ll _\varepsilon (\max\{\abs{\bm{a}}, T\})^\varepsilon .
\]

\end{lemma}

\begin{proof}
Let us set $T' = \max \{ \abs{\bm{a}}, T \}$ and fix $\varepsilon > 0 $.\\
Let us define $W$ to be the set of $(\alpha_1, \dots, \alpha_{n+2}, u_1, \dots, s_n) \in \R^{n+2} \times \Sigma$ such that
\begin{equation}
\label{real_linear_condition_extended_eqn}
\left\{
\begin{aligned}
\alpha_1 u_1 + \dots + \alpha_n u_n &= \alpha_{n+1}, \\
\alpha_1 v_1 + \dots + \alpha_n v_n &= \alpha_{n+2}, \\
\alpha_1 r_1 + \dots + \alpha_n r_n &= 0, \\
\alpha_1 s_1 + \dots + \alpha_n s_n &= 0;
\end{aligned}
\right.
\end{equation}
this is a definable set in $\R_{\mathrm{an}}$.
Using the notation introduced in \eqref{height_stratification_defn_eqn},
we denote by $W^{\sim}(\Q, T')$
the set of tuples $(\alpha_1, \dots , \alpha_{n + 2}, u_1 , \dots, s_n) \in W$ 
with rational $\alpha_1, \dots , \alpha_{n+2}$ of height at most $T'$.
Note that
$\pi_2 (W^{\sim}(\Q, T') ) \supseteq \Sigma^{(1)} (\bm{a}, T )$, where $\pi_2: W \to \Sigma$ is the projection to $\Sigma$.
Then, $\abs{\Sigma^{(1)} (\bm{a}, T )} \leq \abs{\pi_2 (W^{\sim}(\Q, T'))}$.
We claim that $\abs{\pi_2 (W^{\sim}(\Q, T'))} \ll_\varepsilon (T')^\varepsilon $.

Assume by contradiction that this is not the case; then, by Proposition \ref{ominHPprop}, there exists a continuous
definable $\delta: [0, 1] \to W$ such that $\delta_1 := \pi_1 \circ \delta: [0, 1] \to \R^{n + 2}$ is semi-algebraic and
$\delta_2 := \pi_2 \circ \delta: [0, 1] \to \Sigma$ is non-constant. Therefore, there is a connected infinite subset $E \subseteq [0, 1]$
such that $\delta_1 (E )$ is contained in a real algebraic curve and $\delta_2 (E )$ has positive dimension. Then,
there exists a connected infinite $D' \subseteq D$ with $\Theta(D') \subseteq \delta_2 (E )$.\\
The coordinate functions $\alpha_1, \dots , \alpha_{n + 2}$ on $D'$ satisfy $n + 1$ independent 
algebraic relations with coefficients in $\C$. Moreover, the relations given by
\eqref{real_linear_condition_extended_eqn} translate to

\begin{equation*}
\left\lbrace
\begin{aligned}
\alpha_1 z_1 + \dots + \alpha_n z_n &= \alpha_{n+1} f + \alpha_{n+2} g, \\
\alpha_1 \ddl{z_1} + \dots + \alpha_n \ddl{z_n} &= \alpha_{n+1} \ddl{f} + \alpha_{n+2} \ddl{g},
\end{aligned}
\right.
\end{equation*}
adding 2 algebraic relations among the $\alpha_1, \dots, \alpha_{n+2}$ and the $z_j, \ddl{z_j}, f, g, \ddl{f}, \ddl{g}$.\\
Thus, on $D'$, and therefore by continuation on the whole $D$,
the $n + 2 + 2n$ functions $\alpha_1, \dots, \alpha_{n+2}, z_1, \dots, z_n, \ddl{z_1}, \dots, \ddl{z_n}$
satisfy $n+1+2$ independent algebraic relations over $F=\C\left(\lambda, f, g, \ddl{f}, \ddl{g}\right)$. Therefore,
\[ 
\mathrm{trdeg}_F F\left(z_1, \dots, z_n, \ddl{z_1},\dots, \ddl{z_n}\right) \leq 2n - 1.
\]
This contradicts Lemma \ref{funcTrascLemma}, and proves the claim and the lemma.

\end{proof}
\begin{proof}[Proof of Proposition \ref{MainEstProp}]
If $t \in D (\bm{a})$, then $\Theta(t )$ satisfies \eqref{real_linear_condition_eqn}
for some integers $a_{n + 1}$ and $a_{n + 2}$. 
Now, since
$D$ is compact, we have that the sets 
$z_i (D) , f (D) , g (D)$ are bounded, and
therefore we can choose $a_{n + 1}, a_{n + 2}$ bounded solely in terms of $\abs{\bm{a}}$.
Therefore, we have $\Theta(t ) \in \Sigma^{(1)} (\bm{a}, T_0)$, with $T_0 \ll \abs{\bm{a}}$. 
Now, by Lemma \ref{uniform_bound_linear_rel_lemma}, we have
$\abs{D (\bm{a})} \ll \abs{\Sigma^{(1)} (\bm{a}, T_0)}$ and the claim follows from Lemma \ref{Habegger_bound_lemma}.
\end{proof}

\section{Small generators of the relations lattices}
In this section we recall general facts about linear relations on elliptic curves.\\
For a point $(\xi_1, \dots, \xi_N ) \in \overline{\Q}^N$, the absolute logarithmic Weil height $h (\xi_1, \dots, \xi_N )$ is defined by
\[
h (\xi_1, \dots, \xi_N ) = \frac{1}{[\Q(\xi_1, \dots, \xi_N):\Q]} \sum_v \log \max \left\lbrace 1, \abs{\xi_1}_{v}, \dots, \abs{\xi_N}_{v} \right\rbrace,
\]
where $v$ runs over a suitably normalized set of valuations of $\Q(\xi_1, \dots, \xi_N)$.\\

We recall our running setting. We have an irreducible curve $\mathcal{C}$ embedded in the $n$th fibered power of the Legendre elliptic scheme $\legsch^n \to Y(2)$, not contained in a fiber or proper flat subgroup of $\legsch^n$.
We considered the subset $\hatcalC$, obtained by removing from $\mathcal{C}$ finitely many points, and showed that for any point $\bm{c} \in \widehat{\mathcal{C}}$, we can define generators of the period lattice $f(\bm{c}')$ and $g(\bm{c}')$, and the elliptic logarithms $z_i$ of the projections $P_i(\bm{c}') \in \lambda(\bm{c}')$, for $\bm{c}'$ in a neighborhood $N_{\bm{c}_\ast}$ of $\bm{c}$.
Moreover, we can look at such elliptic logarithms as a function of $\lambda$, so it does make sense to write $\ddl{z_i}$.\\

We define, for $\bm{c} \in \widehat{\mathcal{C}}$, the lattice of relations:
\begin{align*}
L(\bm{c}) &= \{(a_1, \dots, a _n) \in \Z^n \st a_1 P_1(\bm{c}) + \dots + a _n P_n(\bm{c}) = O(\lambda(\bm{c})) \in \legsch_{\lambda(\bm{c})} \}\\
&= \{\bm{a} \in \Z^n \st \sum_{i=1}^{n} a_i z_i(\lambda(\bm{c})) = a_{n+1} f(\lambda(\bm{c})) + a_{n+2} g(\lambda(\bm{c})) \text{ for some } a_{n+1}, a_{n+2} \in \Z \}.
\end{align*}
This is a sublattice of $\Z^n$, and it is of some positive rank $r$ whenever $\bm{c}$ is contained in a proper subgroup of $\legsch^n$.\\
We can also define the lattice of \emph{singular} relations:
\begin{align*}
L^{sing}(\bm{c}) &= \{(a_1, \dots, a _n) \in \Z^n \st a_1 P_1(\bm{c}) + \dots + a _n P_n(\bm{c}) = O(\lambda(\bm{c})) \text{ tangentially}\}\\
&=
\begin{multlined}[t] 
\left\lbrace  \bm{a} \in \Z^n \ST 
\exists a_{n+1}, a_{n+2} \in \Z :  
	\sum_{i=1}^{n} a_i z_i(\lambda(\bm{c})) = a_{n+1} f(\lambda(\bm{c})) + a_{n+2} g(\lambda(\bm{c})), \right.  \\
	\left. \sum_{i=1}^{n} a_i \ddl{z_i} (\lambda(\bm{c})) = a_{n+1} \ddl{f}(\lambda(\bm{c})) + a_{n+2} \ddl{g}(\lambda(\bm{c}))
\right\rbrace.
\end{multlined}
\end{align*}

Suppose now that $\bm{c} \in \hatcalC$ is one of the points considered in Theorem \ref{Main_Thm}, i.e. there exits $\bm{a} \in \Z^n\setminus \{0\}$ with
$ \varphi_{\bm{a}}(\bm{c}) = O(\lambda(\bm{c})) $
and $\mathcal{C}$ and $\ker (\varphi_{\bm{a}})$ intersect tangentially at $\bm{c}$.\\
In this situation we have that  the ranks of $L^{sing}(\bm{c})$ and $L(\bm{c})$ are both positive, and $\lambda(\bm{c})$ is an algebraic number, as $\mathcal{C}$ is defined over $\overline{\Q}$ and not contained in any proper algebraic subgroup of $\legsch^n$.
Under this condition, we would like to show that $L^{sing} (\bm{c})$ has a set of
generators with small norm $\abs{\bm{a}} = \max \{ a_1 , \dots, a_n \}$.
Unfortunately we are able to prove such a result only in the case when $\rk(L^{sing}(\bm{c})) = \rk(L(\bm{c}))$, but we will see that this will suffice.\\

By slight abuse of notation we will omit in the following the dependence on $\bm{c}$, writing $\lambda$ for $\lambda(\bm{c})$ and $P_1, \dots, P_n$ for $P_1(\bm{c}), \dots, P_n(\bm{c})$.\\
Suppose that $P_1, \dots, P_n \in \legsch_\lambda$ are 
defined over some finite extension $K$ of $\Q(\lambda)$ of degree $d = [K: \Q]$. 
Furthermore, assume that $P_1, \dots, P_n$ have
Néron–Tate height $\hat{h}$ at most $q \geq1$ (for the definition of Néron–Tate height, see for example
p. 255 of \cite{Masser88}).
In case the $P_i$ are all torsion, we have $\hat{h}(P_i)=0$, and we put $q=1$.\\

We recall the following lemma.

\begin{lemma}[\cite{BarroCapuano16}, Lemma 6.1, based on \cite{Masser88}]
\label{SmallGenerators_lemma}
Under the above hypotheses, there are generators $\bm{a}_1, \dots, \bm{a}_r$ of $L$ with
\[
\abs{\bm{a}_i} \leq \delta_1 d ^{\delta_2} (h (\lambda) + 1)^{2n} q ^{\frac{1}{2}(n - 1)},
\]
for some positive constants $\delta_1, \delta_2$ depending only on $n$.
\end{lemma}

Before stating and proving the next result we notice that, by definition, $L^{sing} \subseteq L$. In general we do not expect them to be equal.

\begin{lemma}
\label{RelationsNoGap_lemma}
One has that
\[
	\Q L^{sing} \cap L = L^{sing};
\]
in particular, if $\rk(L^{sing}) = \rk(L)$, then $L^{sing} = L$.
\end{lemma}
\begin{proof}
Clearly we have $\Q L^{sing} \cap L\supseteq L^{sing} $, so we only have to prove the other inclusion.

Let $\bm{a} \in \Q L^{sing} \cap L$. Since $\bm{a} \in L$, there exist some integers $a_{n+1}, a_{n+2}$ such that $\sum_{i=1}^{n} a_i z_i(\lambda) = a_{n+1} f(\lambda) + a_{n+2} g(\lambda)$, which are unique because $f$ and $g$ are linearly independent over $\R$.\\
Since $\bm{a} \in \Q L^{sing}$, then for some integer $\mu$ we have $\mu \bm{a} \in L^{sing}$, hence for some $b_{n+1}, b_{n+2} \in \Z$, we have
\[
\left\lbrace
\begin{aligned} 
&\sum_{i=1}^{n} \mu a_i z_i(\lambda) - b_{n+1} f(\lambda) - b_{n+2} g(\lambda) = 0\\
&\sum_{i=1}^{n} \mu a_i \ddl{z_i} (\lambda) - b_{n+1} \ddl{f}(\lambda) - b_{n+2} \ddl{g}(\lambda) = 0
\end{aligned}
\right. 
\]
We know, moreover, that $\sum_{i=1}^{n} \mu a_i z_i(\lambda) - \mu a_{i+1} f(\lambda) - \mu a_{i+2} g(\lambda) = 0$; since the $b_i$'s are unique again by the linear independence of $f$ and $g$, we have that $\mu a_i = b_i$ for $i=n+1, n+2$.
Hence we have
\begin{align*}
0 &= \sum_{i=1}^{n} \mu a_i \ddl{z_i} (\lambda) - \mu a_{n+1} \ddl{f}(\lambda) - \mu a_{n+2} \ddl{g}(\lambda)\\
&= \mu \left(\sum_{i=1}^{n} a_i \ddl{z_i} (\lambda) - a_{n+1} \ddl{f}(\lambda) - a_{n+2} \ddl{g}(\lambda) \right)
\end{align*}
so
\[
\sum a_i \ddl{z_i} (\lambda) - a_{n+1} \ddl{f}(\lambda) - a_{n+2} \ddl{g}(\lambda) = 0,
\]
from which we deduce that actually $\bm{a} \in L^{sing}$, as desired.
\end{proof}

Combining the two lemmas we get the following.
\begin{coroll}
\label{SmallGeneratorsSing_lemma}
Let $\bm{c}, \lambda, q, d$ be as in Lemma \ref{SmallGenerators_lemma}; 
assume $\rk (L^{sing}(\bm{c}))=\rk (L(\bm{c}))$. 
Then, there are generators $\bm{a}_1, \dots, \bm{a}_r$ of $L^{sing}(\bm{c})$ with
\[
\abs{\bm{a}_i} \leq \delta_1 d ^{\delta_2} (h (\lambda) + 1)^{2n} q ^{\frac{1}{2}(n - 1)},
\]
for some positive constants $\delta_1, \delta_2$ depending only on $n$.
\end{coroll}

\section{Bounded height}
Let $\mathcal{C}'$ be the set of points $\bm{c}$ of
$\hatcalC$ considered by Theorem \ref{Main_Thm}, i.e. such that $P_1(\bm{c})$, $\dots$, $P_n(\bm{c})$ satisfy
a linear relation with integer coefficients on the specialized elliptic curve $\legsch_{\lambda(\bm{c})}$, and fix $\bm{c}_0 \in \mathcal{C}'$. 
Since $\mathcal{C}$ is defined over $\overline{\Q}$, the
$x_j (\bm{c}_0 )$ and $\lambda(\bm{c}_0 )$ must be algebraic, unless the $P_j$ are identically linearly dependent on $\mathcal{C}$,
which we excluded by hypothesis.

Then, by Silverman’s specialization Theorem
\cite{SilvermanSpec} there exists $\gamma_1 > 0$ such that
\begin{equation}
\label{heightBound}
h(\lambda(\bm{c}_0 )) \leq \gamma_1.
\end{equation}

This bound has some important consequences that we will see in the rest of this section.\\ 
First, we will show a ``Large Galois Orbit'' type of statement. 
Let $k$ be a number field on which $\mathcal{C}$ is defined. 
We can also assume that the finitely
many points we excluded from $\mathcal{C}$ to get 
$\hatcalC$, which are algebraic, are defined over $k$.\\
In the proof of Theorem \ref{Main_Thm} we want to get a bound on the degree of points in $\mathcal{C}^{\{1, sing\}}$, using Proposition \ref{MainEstProp} and the fact that every point in $\mathcal{C}^{\{1, sing\}}$ has many Galois conjugates also in $\mathcal{C}^{\{1, sing\}}$.
However, Proposition \ref{MainEstProp} gives a bound only in a small neighborhood of a point.
The next lemma shows that a good proportion of Galois conjugates are actually in the same neighborhood.\\
Let $\delta > 0$ be a small real number,
and let
\begin{equation*}
\mathcal{C}^\delta = \left\{ \bm{c} \in \mathcal{C} \st \lVert \bm{c} \rVert \leq \frac{1}{\delta}, \lVert \bm{c} - \bm{c}' \rVert \geq \delta \text{ for all } \bm{c} \in \mathcal{C}\setminus \hatcalC \right\},
\end{equation*}
where $\lVert \cdot \rVert$ is the standard norm of $\C^{2n + 1}$.
\begin{lemma}
\label{bound_on_embedding_in_compact_lemma}
There is a positive $\delta $ such that there are at least $\frac{1}{2} [k(\bm{c}_0 ) : k]$ different
$k$-embeddings $\sigma$ of $k(\bm{c}_0 )$ in $\C$ such that $\sigma (
\bm{c}_0 )$ lies in $\mathcal{C}^\delta$ for all $\bm{c}_0 \in \mathcal{C}'$.
\end{lemma}
\begin{proof}
See Lemma 8.2 of \cite{MZ14}.
\end{proof}

\begin{remark}
We would like to point out that it would be possible to avoid the restriction to a smaller domain and the use of the previous lemma by either constructing the Betti coordinates $u_j, v_j$ in a more global way, following \cite{Gao20a}, or by using the results from Peterzil and Starchenko \cite{PeterzilStarchenko2004} and Jones and Schmidt \cite{Jones2020}, who proved that the Weierstrass $\wp$ function, and hence its derivative, can be defined globally in the structure $\R_{\mathrm{an,exp}}$.
\end{remark}

Secondly, we show a bound for the Néron-Tate heights and the degree of the coordinates of the points in $\mathcal{C}^{\{1, sing\}}$, needed in order to apply Corollary \ref{SmallGeneratorsSing_lemma}.
\begin{lemma}[{\cite[Lemma 7.2]{BarroCapuano16}}]
\label{NeronTateHeightBound_lemma}
There exist positive constants $\gamma_2, \gamma_3$ such that, for every $\bm{c}_0 \in \mathcal{C}'$ and every $i = 1,\dots, n $ we have
\[
\hat{h}(P_i (\bm{c}_0)) \leq  \gamma_2,
\]
and the $P_i(\bm{c}_0)$'s are defined over some number field $K \supseteq k(\lambda(\bm{c}_0))$, with
\[
[K:\Q] \leq \gamma_3 [k(\lambda(\bm{c}_0)):k].
\]
\end{lemma}
\begin{remark}
In the assumptions of Lemma 7.2 of \cite{BarroCapuano16} the authors assume that $\bm{c}_0$ satisfies two linearly independent conditions on the specialized elliptic curve $\legsch_{\lambda(\bm{c}_0)}$. However, the proof uses this only to get a height bound like \eqref{heightBound}, so the same proof works in our situation as well.
\end{remark}

\section{Proof of Theorem \ref{Main_Thm}}
\label{Section:proof_of_main}

We want to show that there are at most finitely many $\bm{c}$ on the curve such that
$P_1 (\bm{c}), \dots , P_n (\bm{c})$ satisfy a linear relation with integer coefficients on $\legsch_{\lambda(\bm{c})}$ with multiplicity $\ge 2$. 
Let $\mathcal{C}^{\{1, sing\}}$ be the set of such points.
Since $\mathcal{C} \setminus \hatcalC$ is finite, as observed in Section \ref{Section30}, we only need to prove the finiteness of
\[
\mathcal{C}'':=\mathcal{C}^{\{1, sing\}} \cap \hatcalC .
\]
In the previous section we showed that for $\bm{c} \in \mathcal{C}''\subseteq\mathcal{C}'$ the height of $\lambda(\bm{c})$ is bounded by a constant $\gamma_1$.
By Northcott’s theorem \cite{Northcott49} we then only need to bound the degree $d$ of $\lambda(\bm{c})$ over $k$, a number field over which $\mathcal{C}$ and the points of $\mathcal{C} \setminus \hatcalC$ are defined.

Fix one $\bm{c}_0 \in \mathcal{C}''$ and $d _0 := [k (\bm{c}_0 ) : k ]$ which we suppose large.
It is enough to prove that $d_0$ is bounded, as $\lambda(\bm{c}_0) \in k (\bm{c}_0)$.
By Theorem 1.1 of \cite{BarroCapuano16}, up to removing a finite number of points from $\mathcal{C}''$ we can assume that the $P_i(\bm{c}_0)$'s do not satisfy two linearly independent linear relations, that is, $\rk(L(\bm{c}_0)) = 1$. 
Since they satisfy a linear relation with tangency condition, we also have $\rk(L^{sing}(\bm{c}_0)) = 1$.\\

By Lemma \ref{bound_on_embedding_in_compact_lemma}, we can
choose $\delta$, independently of $\bm{c}_0$, such that $\bm{c}_0$ has at least $\frac{1}{2} d _0$ conjugates in $\mathcal{C}^\delta$.
Now, since $\mathcal{C}^\delta$ is compact, there are 
$\bm{c}_1, \dots, \bm{c}_{\gamma_4} \in \widehat{\mathcal{C}}$ with corresponding neighborhoods 
$N_{\bm{c}_1}, \dots, N_{\bm{c}_{\gamma_4}}$ and 
$D_{\bm{c}_1}, \dots, D_{\bm{c}_{\gamma_4}} \subseteq \pi(\widehat{\mathcal{C}}) $, 
where $D_{\bm{c}_i} \subseteq \lambda (N_{\bm{c}_i} )$ contains $\lambda (\bm{c}_i )$ and is analytically isomorphic to a closed disc, and we have that
the $\lambda^{-1} (D_{\bm{c}_i}) \cap N_{\bm{c}_i}$ for $i = 1, \dots, \gamma_4$ cover $\mathcal{C}^\delta$.

We can assume that $D_{\bm{c}_1}$ contains $t_0^\sigma = \lambda (\bm{c}_0^\sigma )$ for at least $\frac{1}{2\gamma_4} d_0$ conjugates $\bm{c}_0^\sigma $. 
Since each $t \in \lambda(\mathcal{C})$ has a uniformly bounded number of preimages $\bm{c} \in \mathcal{C} $, 
we can suppose that there are at
least $\frac{1}{\gamma_5} d_0$ distinct such $t_0^\sigma$ in $D_{\bm{c}_1}$.\\

Now, the corresponding points $P_1 (\bm{c}_0^\sigma ), \dots, P_n (\bm{c}_0^\sigma )$ satisfy the same relations. 
Moreover, by Corollary \ref{inters_invariant_iso_cor}, the tangency of these relations is also preserved,
so we have that $\rk(L^{sing}(\bm{c}_0^\sigma))=\rk(L(\bm{c}_0^\sigma))=1$.

By Lemma \ref{NeronTateHeightBound_lemma}, $\hat{h}(P_i(\bm{c}_0^\sigma)) \leq \gamma_2$, and the $P_i(\bm{c}_0^\sigma)$ are defined over a field $K$ of degree $\leq \gamma_3 d_0$ over $\Q$.
Hence, by applying Corollary \ref{SmallGeneratorsSing_lemma} and recalling \eqref{heightBound}, there exist $\bm{a} \in L^{sing}(\bm{c}_0^\sigma)$ such that 
\begin{equation}
\label{GeneratorsBound_eq}
\abs{\bm{a}} \leq \gamma_7 d_0^{\gamma_8}.
\end{equation}

Recall that in Section \ref{Section30} we defined on $D_{\bm{c}_1}$ the functions $f, g$ to be generators of the period lattice of $\legsch_\lambda$, and the $z_i$ as the elliptic logarithms of the $P_i$.
Since $\bm{a} \in L^{sing}(\bm{c}_0^\sigma)$, we have
\begin{equation*}
\left\lbrace \begin{aligned}
a_1 z_1(t_0^\sigma) + \dots + a_n z_n(t_0^\sigma) &= a_{n+1} f(t_0^\sigma) + a_{n+2} g(t_0^\sigma), \\
a_1 \ddl{z_1}(t_0^\sigma) + \dots + a_n \ddl{z_n}(t_0^\sigma) &= a_{n+1} \ddl{f}(t_0^\sigma) + a_{n+2} \ddl{g}(t_0^\sigma). \\
\end{aligned} \right.
\end{equation*}
Notice also that the same relationships are satisfied by all the other Galois conjugates of $t_0^\sigma$ in $D_{\bm{c}_1}$, by Corollary \ref{inters_invariant_iso_cor}.\\
In Section \ref{SectionMainEstimate} we defined $D_{\bm{c}_1}(\bm{a})$ as the set of $t \in D_{\bm{c}_1}$ such that relations of the form \eqref{complex_tangency_condition_eqn} (same as the above) hold. 
In particular, we have that $t_0^\sigma \in D_{\bm{c}_1}(\bm{a})$ for all the $t_0^\sigma$ as above.
Therefore, $\abs{D_{\bm{c}_1}(\bm{a})} \geq \frac{1}{\gamma_5} d_0$.\\
On the other hand, by Proposition \ref{MainEstProp} we have $\abs{D_{\bm{c}_1} (\bm{a})} \leq \gamma_9 \abs{\bm{a}}^\varepsilon$, so applying \eqref{GeneratorsBound_eq} we have that 
$\abs{D_{\bm{c}_1} (\bm{a})} \leq \gamma_{10} d_0^{\gamma_8 \varepsilon}$.
Therefore if we choose $\varepsilon = 1 / (2 \gamma_8)$ we have a contradiction if $d_0$ is large enough.

We just deduced that $d_0$ is bounded and, by \eqref{heightBound} and Northcott's Theorem, we have the claim of Theorem \ref{Main_Thm}.

\qed

\section{An application to elliptic divisibility sequences}

In this section we discuss an application of Theorem \ref{Main_Thm} to geometric elliptic divisibility sequences.\\

Let $\mathcal{C}$ be an irreducible curve. A \emph{divisibility sequence} on $\mathcal{C}$ is a sequence of divisors $(D_n)_{n \geq1}$ of $\mathcal{C}$ such that $\forall m, n \in \N$, if $m \mid n$ then $D_m \mid D_n$, that is, $D_n - D_m$ is an effective divisor.
This notion is the function field analogue of an integral divisibility sequence.\\
Let $\pi:\mathcal{G} \to \mathcal{C}$ be a group scheme with zero-section $O$.
As noted by Silverman in \cite{Silverman2005}, every section $P:\mathcal{C} \to \mathcal{G}$ that is not identically torsion, i.e. $nP \neq O$ for all $n\geq 1$, has an associated divisibility sequence given by
\begin{equation}
\label{eqn:geom_seq}
D_{nP} := O^*(nP),
\end{equation}
i.e., $D_{nP}$ is the pull-back along the zero section of the divisor $nP$ on $\mathcal{G}$; such divisibility sequences are called \emph{geometric}.\\
A conjecture of Silverman, stated in \cite{Silverman2005}, predicts that, if $\mathcal{G}$ has relative dimension at least $2$, and the images of the multiples of $P$ are dense in the generic fiber, then
$D_{nP} = D_P$ for infinitely many $n>0$.
This question has been studied by Silverman \cite{Silverman2004} and Ghioca, Hsia and Tucker \cite{Ghioca2018} when $\mathcal{G}$ is the product of two elliptic schemes, and by Barroero, Capuano and Turchet \cite{BCT2024} in the case of split semiabelian schemes.
The results proved in \cite{Ghioca2018} and \cite{BCT2024} are stronger than the one mentioned above: under the conditions predicted by Silverman's conjecture, the set of $n \in \N$ for which $D_{nP} = D_P$ is the complement of a finite union of arithmetic progressions.

Generalizing the setting above, if we have two sections $P$ and $Q$ of $\mathcal{G} \to \mathcal{C}$, we can construct a sequence of divisors by replacing the zero section in \eqref{eqn:geom_seq} with $Q$, provided that $nP \neq Q$ for all $n\geq 1$.
We then get
\begin{equation}
\label{D_nP_Q_eqation}
D_{nP, Q} := Q^*(nP) = \pi_*(nP \cap Q),
\end{equation}
where, by slight abuse of notation, we are identifying $nP$ and $Q$ with their image, and the intersection is intended to be the scheme-theoretic intersection as zero-cycles.\\
These sequences have also been studied by Ghioca, Hsia and Tucker \cite{Ghioca2018} and Barroero, Capuano and Turchet \cite{BCT2024},
for $\mathcal{G}$ the product of two elliptic schemes and a split semiabelian scheme respectively.
They proved that, under the same conditions on $P$ as the above mentioned conjecture of Silverman's, $D_{nP, Q}$ is bounded by a divisor independent of $n$.

If the relative dimension of $\mathcal{G}$ is 1, then there is no expectation that $D_{nP} = D_P$ infinitely often.
However, in this case it is interesting to study when $D_{nP}$ is reduced, that is, it is of the form $D_{nP} = \sum_i t_i$ for $t_i \in \mathcal{C}$ distinct points (equivalently, each non-zero coefficient of $D_{nP}$ is equal to 1).
This condition is an analogue for an integer to be square-free.

In \cite{UU21}, the authors prove that, for $\mathcal{G}=\mathcal{E}$ an elliptic scheme and $P$ non identically torsion, $D_{nP}$ is reduced for all $n$ outside a finite union of arithmetic progressions.
Using Theorem \ref{Main_Thm} we prove that, under the further assumption that $\mathcal{E}$ is not isotrivial, a similar statement holds for $D_{nP, Q}$.
We start by proving a consequence of Theorem \ref{Main_Thm}.

\begin{thm}
\label{Application_thm_1}
Let $\pi: \mathcal{E} \to S$ be a non-isotrivial elliptic scheme over a smooth curve $S$, and let $P, Q$ be sections $S \to \mathcal{E}$.
If there exist infinitely many $s \in S(\overline{\Q})$ for which there exists $m_s \in \Z$ such that $m_s P(s) = Q(s)$ tangentially, then, for some $m \in \Z$, we have that $Q=m P$ identically on $S$.
\end{thm}
\begin{proof}
Consider $(s_n)_n$ an infinite sequence of pairwise distinct $s_n \in S(\overline{\Q})$ for which there exists $m_{s_n} \in \Z$ such that
\begin{equation}
\label{tangent_intersection_section_eqn}
m_{s_n} P(s_n) = Q(s_n) \text{  tangentially}.
\end{equation}
We may also assume that $m_{s_n}$ has the smallest possible absolute value among the integers for which such a relation holds.

If the $m_{s_n}$ are bounded, then, up to taking a subsequence, we have that for a fixed $m$ there exist infinitely many $s \in S(\overline{\Q})$ where $m P(s) = Q(s)$, from which the claim follows.
We will therefore assume from now on that the $m_{s_n}$ are unbounded.

From Theorem \ref{Main_Thm} it follows that, for some $a, b \in \Z$ not both zero,
\begin{equation}
\label{relation_identical_section_eqn}
a P = b Q
\end{equation}
identically on $S$.
Combining this with \eqref{tangent_intersection_section_eqn}, we find that $P(s_n)$ is torsion for infinitely many $n$, since we have
\begin{equation}
\label{p_section_is_torsion_eqn}
O(s_n)=a P(s_n) - b Q(s_n) = a P(s_n) - b m_n P(s_n) = (a-b m_n)P(s_n)
\end{equation}
and $(a- b m_n)$ is non-zero infinitely often, as $m_n$ is unbounded.
Moreover, since \eqref{tangent_intersection_section_eqn} is true tangentially and \eqref{relation_identical_section_eqn} is true identically, we get that \eqref{p_section_is_torsion_eqn} also holds tangentially.
By \cite[Proposition 2.4]{CDMZ21} (see also \cite{UU21}), $P$ must be identically torsion, and so the same is true for $Q$, by \eqref{relation_identical_section_eqn}.

\end{proof}

\begin{thm}
\label{thm:elliptic_div}
Let $\pi:\mathcal{E} \to \mathcal{C}$ be a non-isotrivial elliptic scheme, and let $P, Q:\mathcal{C} \to \mathcal{E}$ be two sections such that no multiple of $P$ is identically equal to $Q$.
Let $D_{nP, Q}$ be as in \eqref{D_nP_Q_eqation}.
Then the set 
\[
\Xi = \{ n \in \N \st D_{nP, Q} \text{ is not reduced} \}
\]
is the union of a finite set and a finite union of arithmetic progressions.
\end{thm}
\begin{proof}
For $s \in \mathcal{C}$ let $\Xi_s = \{ n \in \N \st 2s \leq D_{nP, Q} \}$, that is, the set of $n\in \N$ such that $nP$ and $Q$ intersect tangentially in $s$.
Clearly, $\Xi = \bigcup_{s \in \mathcal{C}} \Xi_s$.
Moreover, by Theorem \ref{Application_thm_1}, there are only finitely many $s \in \mathcal{C}$ such that $\Xi_s$ is nonempty. It is then enough to show that for each $s \in \mathcal{C}$ we have that $\Xi_s$ is either empty, a singleton, or a finite union of arithmetic progressions.
Suppose that $\Xi_s$ is nonempty, and let $n_0$ be its smallest element. For any other $n \in \Xi_s$ we have that both $n_0 P$ and $n P$ intersect $Q$ tangentially over $s$. Hence, $(n - n_0) P$ intersects the zero section tangentially over $s$. 
By \cite[Thm 1.4]{UU21}, we can deduce that $\Xi_s - n_0$ is a finite union of arithmetic progressions starting from 0, so $\Xi_s$ itself is a finite union of arithmetic progressions starting from $n_0$.
\end{proof}
\begin{remark}
In a similar way as Theorem \ref{abelian_main_thm}, we don't expect the condition of $\mathcal{E}$ being non-isotrivial to be necessary. 
Indeed, having an isotrivial analogue of Theorem \ref{Main_Thm} would suffice, following the same proof as above.
Such a theorem is part of a work in progress with Ballini and Capuano.
\end{remark}
\begin{remark}
The result of Theorem \ref{thm:elliptic_div} is much weaker than what Ulmer and Urz{\'u}a get in the special case $Q = O$ (\cite[Thm 1.4]{UU21}). 
There, the authors prove that $\{ n \in \N \st D_{nP} \text{ is not reduced} \}$ is the finite union of arithmetic progressions \emph{starting from 0}, implying that if $D_{P}$ is reduced, then $D_{nP}$ is reduced infinitely often.
In our case, the arithmetic progressions may not start from 0, so Theorem \ref{thm:elliptic_div} does not provide an easy method to check wether $D_{nP, Q}$ is reduced infinitely often.
\end{remark}

\section*{Acknowledgments}
We would like to thank Laura Capuano for many
useful discussions and comments, and Daniel Bertrand for his suggestions for the proof of Lemma \ref{funcTrascLemma_fixed_point}, and David Masser for the interest in this work and for suggesting potential future applications.\\
We also thank Nelson Alvarado, Luca Ferrigno, Federico Pieroni and Roberto Vacca, for helpful discussions.\\
We would also like to thank Harry Schmidt, and the whole Mathematics department of the University of Warwick, for hosting the author during part of the process of writing the present article.\\
The author was supported by the National Group for Algebraic
and Geometric Structures, and their Applications (GNSAGA INdAM).

\bibliographystyle{alpha-abbr3}
\bibliography{biblio}

\end{document}